\documentclass[11pt,twoside,reqno]{amsart}

\usepackage{amssymb, latexsym}
\usepackage[italian,english]{babel}
\usepackage{mathrsfs}
\usepackage{amsmath,amsfonts,amsthm}
\usepackage{paralist}
\usepackage[top=2.50cm, bottom=2.50cm, left=2.50cm, right=2.50cm]{geometry}

\usepackage{color}

\numberwithin{equation}{section}

\usepackage{xspace}
\usepackage[bookmarksnumbered,colorlinks]{hyperref}
\usepackage{graphics}
\def\bb#1\eb{\textcolor{blue}
{#1}} %
\def\br#1\er{\textcolor{red}
{#1}} %
\def\bv#1\ev{\textcolor{green}
{#1}} %
\def\bc#1\ec{\textcolor{cyan}
{#1}} %

\usepackage{graphics}

\def\Xint#1{\mathchoice
  {\XXint\displaystyle\textstyle{#1}}%
  {\XXint\textstyle\scriptstyle{#1}}%
  {\XXint\scriptstyle\scriptscriptstyle{#1}}%
  {\XXint\scriptscriptstyle\scriptscriptstyle{#1}}%
  \!\int}
\def\XXint#1#2#3{{\setbox0=\hbox{$#1{#2#3}{\int}$}
  \vcenter{\hbox{$#2#3$}}\kern-.5\wd0}}
\def\-int{\Xint -}

\newcommand{\R}{\mathbb{R}}

\DeclareMathOperator{\X}{\mathbb{H}}

\DeclareMathOperator{\e}{\varepsilon}
\newcommand{\M}{\mathcal{M}}
\newcommand{\N}{\mathcal{N}}
\DeclareMathOperator{\J}{\mathcal{J}}

\newtheorem{lem}{Lemma}[section]
\newtheorem{thm}{Theorem}[section]

\newtheorem{cor}{Corollary}[section]

\begin{document}
\title[Multiplicity for fractional Schr\"odinger systems]{Multiplicity of solutions for fractional  Schr\"odinger  systems in $\R^{N}$} 

\author[V. Ambrosio]{Vincenzo Ambrosio}
\address{Vincenzo Ambrosio\hfill\break\indent
Dipartimento di Ingegneria Industriale e Scienze Matematiche \hfill\break\indent
Universit\`a Politecnica delle Marche\hfill\break\indent
Via Brecce Bianche, 12\hfill\break\indent
60131 Ancona (Italy)}
\email{v.ambrosio@univpm.it}


\begin{abstract}
In this paper we deal with the following nonlocal systems of fractional Schr\"odinger equations
\begin{equation*}
\left\{
\begin{array}{ll}
\e^{2s} (-\Delta)^{s}u+V(x)u=Q_{u}(u, v)+\gamma H_{u}(u, v) &\mbox{ in } \R^{N}\\
\e^{2s} (-\Delta)^{s}v+W(x)v=Q_{v}(u, v)+\gamma H_{v}(u, v) &\mbox{ in } \R^{N} \\
u, v>0 &\mbox{ in } \R^{N},
\end{array}
\right.
\end{equation*}
where $\e>0$, $s\in (0, 1)$, $N>2s$, $(-\Delta)^{s}$ is the fractional Laplacian, $V:\R^{N}\rightarrow \R$ and $W:\R^{N}\rightarrow \R$ are continuous potentials, $Q$ is a homogeneous $C^{2}$-function with subcritical growth, $\gamma\in \{0, 1\}$ and $H(u, v)=\frac{2}{\alpha+\beta}|u|^{\alpha} |v|^{\beta}$ with $\alpha, \beta\geq 1$ such that $\alpha+\beta=2^{*}_{s}$. \\
We investigate the subcritical case $(\gamma=0)$ and the critical case $(\gamma=1)$, and using Ljusternik-Schnirelmann theory, we relate the number of solutions with the topology of the set where the potentials $V$ and $W$ attain their minimum values.
\end{abstract}

\maketitle
\section{Introduction}

\noindent
In the last decade a tremendous popularity has received the study of nonlinear partial differential equations involving fractional and nonlocal operators, due to the fact that such operators have great applications in many areas of the research such as crystal dislocation, finance, phase transitions, material sciences, chemical reactions, minimal surfaces; see for instance \cite{DPV, MBRS} for more details.\\
Motivated by the interest shared by the mathematical community in this topic, 
the aim of this paper is to investigate the existence and multiplicity of positive solutions for the following nonlinear fractional Schr\"odinger system
\begin{equation}\label{Pgamma}
\left\{
\begin{array}{ll}
\e^{2s} (-\Delta)^{s}u+V(x)u=Q_{u}(u, v)+\gamma H_{u}(u, v) &\mbox{ in } \R^{N}\\
\e^{2s} (-\Delta)^{s}u+W(x)v=Q_{v}(u, v)+\gamma H_{v}(u, v) &\mbox{ in } \R^{N} \\
u, v>0 &\mbox{ in } \R^{N},
\end{array}
\right.
\end{equation}
where $\e>0$ is a parameter, $s\in (0, 1)$, $N>2s$, $V:\R^{N}\rightarrow \R$ and $W:\R^{N}\rightarrow \R$ are continuous potentials, $Q$ is a homogeneous $C^{2}$-function with subcritical growth, $\gamma\in \{0, 1\}$, and $H(u, v)=\frac{1}{\alpha+\beta} |u|^{\alpha} |v|^{\beta}$ where $\alpha, \beta\geq 1$ are such that $\alpha+\beta=2^{*}_{s}=\frac{2N}{N-2s}$.\\
The nonlocal operator $(-\Delta)^{s}$ is the so-called fractional Laplacian operator which can be defined for any $u: \R^{N}\rightarrow \R$ smooth enough, by setting
$$
(-\Delta)^{s}u(x)=-\frac{C(N, s)}{2} \int_{\R^{N}} \frac{u(x+y)+u(x-y)-2u(x)}{|y|^{N+2s}} dy  \quad (x\in \R^{N}),
$$
where $C(N, s)$ is a dimensional constant depending only on $N$ and $s$; see  for instance \cite{DPV}.\\
In the scalar case, problem \eqref{Pgamma} becomes the well-known fractional Schr\"odinger equation
\begin{equation}\label{FSE}
\e^{2s} (-\Delta)^{s}u+V(x)u=f(x, u) \mbox{ in } \R^{N}.
\end{equation}
We recall that one of the main reasons of studying (\ref{FSE}), is related to the seek of standing wave solutions $\Phi(t, x)=u(x) e^{-\frac{\imath c t}{\hbar}}$ for the following time-dependent fractional Schr\"odinger equation
\begin{equation}\label{TDFSE}
i \hbar \frac{\partial \Phi}{\partial t} = \frac{\hbar^{2}}{2m} (-\Delta)^{s} \Phi +
V(x) \Phi - g(|\Phi|)\Phi \mbox{ for } (t, x)\in \R\times \R^{N}. 
\end{equation}
Equation (\ref{TDFSE}) has been proposed by Laskin \cite{Laskin2}, and it is a fundamental equation of fractional Quantum Mechanics in the study of particles on stochastic fields modelled by L\'evy processes.\\
When $s=1$, equation (\ref{FSE}) reduces to the classical Schr\"odinger equation
\begin{align}\label{Pe}
-\varepsilon ^{2} \Delta u + V(x) u = f(u) \mbox{ in } \R^{N},
\end{align}
which has been extensively studied in the last thirty years by many authors; see for instance \cite{AF, BW, CL, DF, FW, Rab, w} and the references therein.

\smallskip

Recently, the study of fractional Schr\"odinger equations has attracted the attention of many mathematicians.
Felmer et al. \cite{FQT} investigated existence, regularity and qualitative properties of positive solution to 
\eqref{FSE} when $V$ is constant, and $f$ is a smooth function with subcritical growth satisfying the Ambrosetti-Rabinowitz condition. Secchi \cite{Secchi1} proved an existence result for a nonlinear fractional Schr\"odinger equation involving a subcritical nonlinearity and under weak assumptions on the behaviour of the potential $V$ at infinity. 
Frank et al. \cite{FLS} studied uniqueness and nondegeneracy of ground state solutions to \eqref{FSE} with $
f(u)=|u|^{\alpha}u$, for all $H^{s}$-admissible powers $\alpha \in (0, \alpha^{*})$.
The author \cite{A4} showed the existence of infinitely many solutions  to \eqref{FSE} with $V(x)=1$, and $f$ is autonomous and satisfies Berestycki-Lions type assumptions.
Shang et al. \cite{SZY} used variational methods to deal with the multiplicity of solutions of a fractional Schr\"odinger equation with critical growth, and with a continuous and positive potential $V$.
Figueiredo and Siciliano \cite{FS} obtained a multiplicity result  by means of the Ljusternik-Shnirelmann and Morse theories for (\ref{FSE}) involving a superlinear nonlinearity with subcritical growth.
Alves and Miyagaki in \cite{AM} dealt with the existence and the concentration of positive solutions to (\ref{FSE}) via penalization technique and the extension method \cite{CS}. 
We also mention the papers \cite{A2, A3, AI, DDPW, DMV, FS, I} where the existence and the multiplicity of solutions to \eqref{FSE} have been investigated under various assumptions on the potential $V$ and the nonlinearity $f$, by using suitable variational and topological methods. \\
Particularly motivated by the papers \cite{FS, SZY}, in this work we aim to extend the multiplicity results for both subcritical and critical cases obtained for the scalar equation \eqref{FSE} to the case of the systems. 
More precisely, we generalize in the nonlocal setting some existence and multiplicity results appeared in \cite{AFS, AS, AY,  FF} in which the authors studied elliptic systems of the type
\begin{equation*}
\left\{
\begin{array}{ll}
-\e^{2} \Delta u+V(x)u=Q_{u}(u, v)+\gamma H_{u}(u, v) &\mbox{ in } \R^{N}\\
-\e^{2} \Delta v+W(x)v=Q_{v}(u, v)+\gamma H_{v}(u, v) &\mbox{ in } \R^{N} \\
u, v>0  &\mbox{ in } \R^{N}.
\end{array}
\right.
\end{equation*}
To the best of our knowledge, there are few results on the nonlinear systems involving the fractional Laplacian in the literature \cite{cdeng, LMBZ, lm, ww} and the results presented here seems to be new in the nonlocal framework.

\smallskip
In order to state the main theorems obtained in this work, we come back to our problem \eqref{Pgamma}, and we introduce the assumptions on the potentials $V$, $W$ and the function $Q$.\\
Firstly, we define the following constants
$$
V_{0}=\inf_{x\in \R^{N}} V(x) \mbox{ and } W_{0}=\inf_{x\in \R^{N}} W(x)
$$
and
$$
V_{\infty}=\liminf_{|x|\rightarrow  \infty} V(x) \mbox{ and } W_{\infty}=\liminf_{|x|\rightarrow  \infty} W(x).
$$
Along the paper, we will assume the following conditions on $V$ and $W$:
\begin{compactenum}[$(H1)$]
\item $V_{0}=W_{0}>0$, and $M=\{x\in \R^{N}: V(x)=W(x)=V_{0} \}$ is nonempty;
\item $V_{0}<\max\{V_{\infty}, W_{\infty}\}$.
\end{compactenum}

\medskip

Regarding the function $Q$, we suppose that $Q\in C^{2}(\R^{2}_{+}, \R)$ and satisfies the following conditions:
\begin{compactenum}[$(Q1)$]
\item there exists $q\in (2, 2^{*}_{s})$ such that $Q(t u, tv)=t^{q}Q(u, v)$ for all $t>0$, $(u, v)\in \R^{2}_{+}$;
\item there exists $C>0$ such that $|Q_{u}(u, v)|+|Q_{v}(u, v)|\leq C(u^{q-1}+v^{q-1})$ for all $(u, v)\in \R^{2}_{+}$;
\item $Q_{u}(0, 1)=0=Q_{v}(1, 0)$;
\item $Q_{u}(1, 0)=0=Q_{v}(0, 1)$;
\item $Q_{uv}(u, v)>0$ for all $(u, v)\in \R^{2}_{+}$.
\end{compactenum}

\smallskip

Since we look for positive solutions of \eqref{Pgamma}, we extend the function $Q$ to the whole $\R^{2}$ by setting $Q(u, v)=0$ if $u\leq 0$ or $v\leq 0$. We note that the $q$-homogeneity of $Q$ implies that the following identity holds:
\begin{equation}\label{2.1}
qQ(u, v)=uQ_{u}(u, v)+vQ_{v}(u, v) \mbox{ for any } (u, v)\in \R^{2}.
\end{equation}
Moreover, using $(Q2)$, we can see that there exists $C>0$ such that 
\begin{equation}\label{2.2}
|Q(u, v)|\leq C(|u|^{q}+|v|^{q}) \mbox{ for any } (u, v)\in \R^{2}.
\end{equation}
A typical example (see \cite{DMFS}) of function $Q$ which satisfies the above assumptions is the following one. 
Let $p\geq 1$ and 
$$
\mathcal{P}_{p}(u, v)=\sum_{\alpha_{i}+\beta_{i}=p} a_{i} u^{\alpha_{i}} v^{\beta_{i}},
$$
where $i\in \{1, \dots, k\}$, $\alpha_{i}, \beta_{i}\geq 1$ and $a_{i}\in \R$. The following functions and their possible combinations, with appropriate choice of the coefficients $a_{i}$, satisfy assumptions $(Q1)$-$(Q5)$ on $Q$
$$
Q_{1}(u, v)=\mathcal{P}_{q}(u, v),  \quad Q_{2}(u, v))=\sqrt[r]{\mathcal{P}_{l}(u, v)} \quad \mbox{ and } \quad Q_{3}(u, v)=\frac{\mathcal{P}_{l_{1}}(u, v)}{\mathcal{P}_{l_{2}}(u, v)},
$$
with $r= l q$ and $l_{1}-l_{2}=q$.

\smallskip

\noindent
Now, we pass to state our main multiplicity results related to \eqref{Pgamma}.
When we take $\gamma=0$ in \eqref{Pgamma}, we have to deal with a system with subcritical growth, namely
\begin{equation}\label{P}
\left\{
\begin{array}{ll}
\e^{2s} (-\Delta)^{s}u+V(x)u=Q_{u}(u, v) &\mbox{ in } \R^{N}\\
\e^{2s} (-\Delta)^{s}v+W(x)v=Q_{v}(u, v) &\mbox{ in } \R^{N} \\
u, v>0 &\mbox{ in } \R^{N}.
\end{array}
\right.
\end{equation}

\noindent
Since we aim to relate the number of solutions of \eqref{P} with the topology of the set $M$ of minima of the potential, it is worth recalling that if $Y$ is a given closed set of a topological space $X$, we denote by $cat_{X}(Y)$ the Ljusternik-Schnirelmann category of $Y$ in $X$, that is the least number of closed and contractible sets in $X$ which cover $Y$; see \cite{W} for more details.\\
With the above notations, the first main multiplicity result can be stated as follows.
\begin{thm}\label{thm1}
Assume that $(H1)$-$(H2)$ and $(Q1)$-$(Q5)$ hold. Then, for any $\delta>0$, there exists $\e_{\delta}>0$ such that for any $\e\in (0, \e_{\delta})$, system \eqref{P} admits at least $cat_{M_{\delta}}(M)$ solutions, where $M_{\delta}=\{x\in \R^{N}: dist(x, M)\leq \delta\}$.
\end{thm}

\noindent
It is worth noting that, a common approach to deal with fractional nonlocal problems, is to make use of the  Caffarelli-Silvestre method \cite{CS}, which consists in transforming via a Dirichlet-Neumann map, a given nonlocal problem into a local degenerate elliptic problem set in the half-space $\R^{N+1}_{+}$ and with a nonlinear Neumann boundary condition. In this work, we prefer to analyze the problem directly in $H^{s}(\R^{N})$ in order to borrow some ideas developed in the case $s=1$ taking care of the fact that in our situation a more careful analysis is needed due to the nonlocal character of $(-\Delta)^{s}$.

\noindent
The proof of Theorem \ref{thm1} is variational and it is based on the method of the Nehari manifold. After proving some compactness results for the functional associated to \eqref{P}, and observing that the level of compactness are deeply related to the behaviour of the potentials $V$ and $W$ at infinity, 
we use some arguments developed in \cite{BC, CL}, to compare the category of some sub-levels of the functional and the category of the set $M$. 

\smallskip

\noindent
In the second part of our paper, we consider the critical case $\gamma=1$, that is
\begin{equation}\label{CP}
\left\{
\begin{array}{ll}
\e^{2s} (-\Delta)^{s}u+V(x)u=Q_{u}(u, v)+\frac{2\alpha}{\alpha+\beta} |u|^{\alpha-2}u |v|^{\beta} &\mbox{ in } \R^{N}\\
\e^{2s} (-\Delta)^{s}v+W(x)v=Q_{v}(u, v)+\frac{2\beta}{\alpha+\beta} |u|^{\alpha} |v|^{\beta-2}v &\mbox{ in } \R^{N} \\
u, v>0 &\mbox{ in } \R^{N},
\end{array}
\right.
\end{equation}
where $\alpha, \beta\geq 1$ are such that $\alpha+\beta=2^{*}_{s}$. \\
In this context, we assume that $Q$ fulfills the following technical assumption: 
\begin{compactenum}[$(Q6)$]
\item $Q(u, v)\geq \lambda u^{\tilde{\alpha}} v^{\tilde{\beta}}$ for any $(u, v)\in \R_{+}^{2}$ with $1<\tilde{\alpha}, \tilde{\beta}<2^{*}_{s}$, $\tilde{\alpha}+\tilde{\beta}=q_{1}\in (2, 2^{*}_{s})$, and $\lambda$ satisfying 
\begin{itemize}
\item $\lambda>0$ if either $N\geq 4s$, or $2s<N<4s$ and $2^{*}_{s}-2<q_{1}<2^{*}_{s}$;
\item $\lambda$ is sufficiently large if  $2s<N<4s$ and $2<q_{1}\leq 2^{*}_{s}-2$.
\end{itemize}
\end{compactenum}

\noindent
To obtain the multiplicity of positive solutions to \eqref{CP}, we proceed as in the subcritical case. Clearly, the lack of the compactness due to the presence of the critical Sobolev exponent, creates a further difficulty, and more accurate estimates are needed to localize the energy levels where the Palais-Smale condition fails.
To circumvent this hitch, we combine the estimates obtained in \cite{SV} with some adaptations of the calculations done in \cite{AFS}, which allow us to prove that the number
\begin{equation*}
\widetilde{S}_{*}(\alpha, \beta)=\inf_{u, v\in H^{s}(\R^{N})\setminus\{(0, 0)\}} \frac{\int_{\R^{N}} |(-\Delta)^{\frac{s}{2}} u|^{2}+|(-\Delta)^{\frac{s}{2}} v|^{2} dx}{\left(\int_{\R^{N}} |u|^{\alpha}|v|^{\beta} dx\right)^{\frac{2}{2^{*}_{s}}}}
\end{equation*}
is strongly related to the best constant $S_{*}$ of the Sobolev embedding $H^{s}(\R^{N})$ into $L^{2^{*}_{s}}(\R^{N})$, and plays a fundamental role when we have to study critical systems like \eqref{CP}.\\
Our second main result can be stated as follows.  
\begin{thm}\label{thm2}
Let us assume that $(H1)$-$(H2)$ and $(Q1)$-$(Q6)$ hold. If $\alpha, \beta\in [1, 2^{*}_{s})$ are such that $\alpha+\beta=2^{*}_{s}$, then  for any $\delta>0$, there exists $\e_{\delta}>0$ such that for any $\e\in (0, \e_{\delta})$, system \eqref{CP} possesses at least $cat_{M_{\delta}}(M)$ solutions.
\end{thm}

We conclude this introduction observing that our results complement  the ones obtained in \cite{FS, SZY}, in the sense that now we are considering the multiplicity results in the case of systems. 
\smallskip

\noindent
The structure of the paper is the following. In Section $2$ we give some preliminary facts about the fractional Sobolev spaces and we set up the variational framework. In Section $3$ we deal with the autonomous problem related to \eqref{P}. In Section $4$ we prove some compactness results for the functional associated with \eqref{P}. In Section $5$ we present the proof of Theorem \ref{thm1}. In the last section, we discuss the existence and the multiplicity of solutions for the system \eqref{Pgamma} in the critical case $\gamma=1$.

\section{preliminaries and variational setting}

In this section we collect some preliminary results about the fractional Sobolev spaces, and we introduce the functional setting.\\
For any $s\in (0,1)$ we define $\mathcal{D}^{s, 2}(\R^{N})$ as the completion of $C^{\infty}_{0}(\R^{N})$ with respect to
$$
\int_{\R^{N}} |(-\Delta)^{\frac{s}{2}} u|^{2} dx =\iint_{\R^{2N}} \frac{|u(x)-u(y)|^{2}}{|x-y|^{N+2s}} \, dx \, dy,
$$
where the above equality holds up to a positive constant, or equivalently
$$
\mathcal{D}^{s, 2}(\R^{N})=\left\{u\in L^{2^{*}_{s}}(\R^{N}): \int_{\R^{N}} |(-\Delta)^{\frac{s}{2}} u|^{2} dx<\infty\right\}.
$$
Let us introduce the fractional Sobolev space
$$
H^{s}(\R^{N})= \left\{u\in L^{2}(\R^{N}) : \int_{\R^{N}} |(-\Delta)^{\frac{s}{2}} u|^{2} dx<\infty \right \}
$$
endowed with the natural norm 
$$
\|u\|_{H^{s}(\R^{N})} = \sqrt{\int_{\R^{N}} |(-\Delta)^{\frac{s}{2}} u|^{2} dx + \int_{\R^{N}} |u|^{2} dx}.
$$

\noindent
For the convenience of the reader, we recall the following embeddings:
\begin{thm}\cite{DPV}\label{Sembedding}
Let $s\in (0,1)$ and $N>2s$. Then there exists a sharp constant $S_{*}=S(N, s)>0$
such that for any $u\in H^{s}(\R^{N})$
\begin{equation}\label{FSI}
\left(\int_{\R^{N}} |u|^{2^{*}_{s}} dx\right)^{\frac{2}{2^{*}_{s}}}  \leq S_{*} \int_{\R^{N}} |(-\Delta)^{\frac{s}{2}} u|^{2} dx . 
\end{equation}
Moreover, $H^{s}(\R^{N})$ is continuously embedded in $L^{q}(\R^{N})$ for any $q\in [2, 2^{*}_{s}]$ and compactly in $L^{q}_{loc}(\R^{N})$ for any $q\in [1, 2^{*}_{s})$. 
\end{thm}

\noindent
We also have a Lions-compactness type lemma.
\begin{lem}\cite{FQT}\label{lionslemma}
Let $N>2s$. If $(u_{n})$ is a bounded sequence in $H^{s}(\R^{N})$ and if
$$
\lim_{n \rightarrow \infty} \sup_{y\in \R^{N}} \int_{B_{R}(y)} |u_{n}|^{2} dx=0
$$
for some $R>0$,
then $u_{n}\rightarrow 0$ in $L^{t}(\R^{N})$ for all $t\in (2, 2^{*}_{s})$.
\end{lem}

\noindent
Now, we give the variational framework of problem \eqref{P}.
Using the change of variable $x\mapsto \e x$, we are led to consider the following problem
\begin{equation}\label{P'}
\left\{
\begin{array}{ll}
 (-\Delta)^{s}u+V(\e x)u=Q_{u}(u, v)  &\mbox{ in } \R^{N}\\
 (-\Delta)^{s}v+W(\e x) v=Q_{v}(u, v) &\mbox{ in } \R^{N} \\
u, v>0 &\mbox{ in } \R^{N}.
\end{array}
\right.
\end{equation}
For any $\e>0$, we introduce the fractional space
$$
\X_{\e}=\{(u, v)\in H^{s}(\R^{N})\times H^{s}(\R^{N}): \int_{\R^{N}} (V(\e x) |u|^{2}+W(\e x) |v|^{2}) dx<\infty\},
$$
endowed with the norm
$$
\|(u, v)\|^{2}_{\e}=\int_{\R^{N}} |(-\Delta)^{\frac{s}{2}} u|^{2}+ |(-\Delta)^{\frac{s}{2}} v|^{2} dx+\int_{\R^{N}} (V(\e x) u^{2}+W(\e x) v^{2}) dx.
$$
Let us introduce
$$
\J_{\e}(u)=\frac{1}{2}\|(u, v)\|_{\e}^{2}-\int_{\R^{N}} Q(u, v) dx
$$
for any $(u, v)\in \X_{\e}$.
We define the minimax level
$$
c_{\e}=\inf_{(u, v)\in \N_{\e}} \J_{\e}(u, v),
$$
where
$$
\N_{\e}=\{(u, v)\in \X_{\e}\setminus \{(0, 0)\}:  \langle \J'_{\e}(u, v),(u, v)\rangle=0\}.
$$
It is standard to check that $\J_{\e}$ possesses a mountain pass geometry. Indeed, $\J_{\e}\in C^{1}(\X_{\e}, \R)$ and $\J_{\e}(0, 0)=0$. 
Using \eqref{2.2} and Theorem \ref{Sembedding}, we get  for any $(u, v)\in \X_{\e}$
$$
\J_{\e}(u, v)\geq \frac{1}{2} \|(u, v)\|^{2}_{\e}-C\|(u, v)\|^{q}_{\e},
$$
so there exist $\mu$, $\rho>0$ such that $\J_{\e}(u, v)\geq \rho$ for $\|(u, v)\|_{\e}=\mu$. From $(Q1)$, we can see that for any $(u, v)\in \X_{\e}\setminus \{(0, 0)\}$
$$
\J_{\e}(t u, t v)=\frac{t^{2}}{2}\|(u, v)\|^{2}_{\e}-t^{q} \int_{\R^{N}} Q(u, v) dx\rightarrow -\infty \mbox{ as } t\rightarrow \infty.
$$
Finally, in view of \eqref{2.2}, we can note that there exists $r>0$ such that for any $\e>0$
\begin{equation}\label{2.3}
\|(u, v)\|_{\e}\geq r \mbox{ for any } (u, v)\in \N_{\e}.
\end{equation}
Since $\J_{\e}$ satisfies mountain pass geometry, we can use the homogeneity of $Q$ to
prove that $c_{\e}$ can be alternatively characterized by
$$
c_{\e}=\inf_{\gamma\in \Gamma_{\e}} \max_{t\in [0, 1]} \J_{\e}(\gamma(t))=\inf_{(u, v)\in \X_{\e}\setminus \{(0, 0)\}} \max_{t\geq 0} \J_{\e}(t u, tv)>0,
$$
where $\Gamma_{\e}=\{\gamma\in C([0, 1], \X_{\e}): \gamma(0)=0, \J_{\e}(\gamma(1))<0\}$.
Moreover, for any $(u, v)\neq (0, 0)$, there exists a unique $t>0$ such that $(tu, tv)\in \N_{\e}$. The maximum of the function $t\rightarrow \J_{\e}(t u, tv)$ for $t\geq 0$ is achieved at $t=\bar{t}$; for more details see \cite{W}.

\section{the autonomous problem when $\gamma=0$}
In this section we establish an existence result for the autonomous problem associated with \eqref{P}.
Let us consider the following subcritical autonomous system
\begin{equation}\label{P0}
\left\{
\begin{array}{ll}
 (-\Delta)^{s}u+V_{0}u=Q_{u}(u, v)  &\mbox{ in } \R^{N}\\
 (-\Delta)^{s}v+W_{0}v=Q_{v}(u, v) &\mbox{ in } \R^{N} \\
u, v>0 &\mbox{ in } \R^{N}.
\end{array}
\right.
\end{equation}
We set $\X_{0}=H^{s}(\R^{N})\times H^{s}(\R^{N})$ endowed with the norm
$$
\|(u, v)\|^{2}_{0}=\int_{\R^{N}} |(-\Delta)^{\frac{s}{2}} u|^{2}+ |(-\Delta)^{\frac{s}{2}} v|^{2} dx+\int_{\R^{N}} (V_{0} u^{2}+W_{0} v^{2}) dx. 
$$
Let us introduce the functional $\J_{0}: \X_{0}\rightarrow \R$ defined as 
$$
\J_{0}(u, v)=\frac{1}{2}\|(u, v)\|^{2}_{0}-\int_{\R^{N}} Q(u, v) dx. 
$$
Let
$$
c_{0}=\inf_{(u, v)\in \N_{0}} \J_{0}(u, v)=\inf_{(u, v)\in X_{0}\setminus \{(0, 0)\}} \max_{t\geq 0} \J_{0}(t u, tv),
$$
where
$$
\N_{0}=\{(u, v)\in \X_{0}\setminus \{(0, 0)\}:  \langle \J'_{0}(u, v),(u, v)\rangle=0\}.
$$

We begin by proving the following useful lemma.
\begin{lem}\label{lem2.1}
Let $\{(u_{n}, v_{n})\}\subset \X_{0}$ be a sequence such that $\J'_{0}(u_{n}, v_{n})\rightarrow 0$ and $(u_{n}, v_{n})\rightharpoonup (0,0)$.
Then we have either
\begin{compactenum}[(i)]
\item $\|(u, v)\|_{0}\rightarrow 0$, or
\item there exist a sequence $(y_{n})\subset \R^{N}$ and $R, \gamma>0$ such that 
$$
\liminf_{n\rightarrow \infty}\int_{B_{R}(y_{n})} (|u_{n}|^{2}+|v_{n}|^{2}) dx\geq \gamma.
$$
\end{compactenum}
\end{lem}
\begin{proof}
Assume that $(ii)$ is not true. Then, for any $R>0$, we get
$$
\lim_{n\rightarrow \infty} \sup_{y\in \R^{N}}\int_{B_{R}(y)} |u_{n}|^{2} dx=0=\lim_{n\rightarrow \infty} \sup_{y\in \R^{N}}\int_{B_{R}(y)} |v_{n}|^{2} dx.
$$
By Lemma \ref{lionslemma}, we can deduce that 
$$
u_{n}, v_{n}\rightarrow 0 \mbox{ in } L^{t}(\R^{N}) \quad  \forall t\in (2, 2^{*}_{s}).
$$
This fact and \eqref{2.2} give 
\begin{equation}\label{ff100}
\int_{\R^{N}} Q(u_{n}, v_{n}) dx \rightarrow 0.
\end{equation} 
Hence, using $\langle \J'_{0}(u_{n}, v_{n}), (u_{n}, v_{n}) \rangle \rightarrow 0$, \eqref{2.1} and \eqref{ff100} we obtain
$$
\|(u_{n}, v_{n})\|^{2}_{0}=\int_{\R^{N}} (Q_{u}(u_{n}, v_{n})u_{n}+Q_{v}(u_{n}, v_{n})v_{n}) dx+o_{n}(1)=q\int_{\R^{N}} Q(u_{n}, v_{n}) dx+o_{n}(1)=o_{n}(1),
$$
which implies that $(i)$ holds.
\end{proof}

\begin{thm}\label{prop2.2}
The problem \eqref{P0} admits a weak solution.
\end{thm}
\begin{proof}
It is clear that $\J_{0}$ has a Mountain Pass geometry, so, in view of Theorem $1.15$ in \cite{W}, we can find a sequence  $\{(u_{n}, v_{n})\}\subset \X_{0}$ such that 
$$
\J_{0}(u_{n}, v_{n})\rightarrow c_{0} \mbox{ and } \J'_{0}(u_{n}, v_{n})\rightarrow 0.
$$
By \eqref{2.1}, we can see that 
\begin{align*}
c_{0}+o_{n}(1)\|(u_{n}, v_{n})\|_{0}&=\J_{0}(u_{n}, v_{n})-\frac{1}{q}\langle \J'_{0}(u_{n}, v_{n}), (u_{n}, v_{n}) \rangle \\
&=\left(\frac{1}{2}-\frac{1}{q}\right)\|(u_{n}, v_{n})\|^{2}_{0},
\end{align*}
which implies that $\{(u_{n}, v_{n})\}$ is bounded in $\X_{0}$. Consequently, thanks to Theorem \ref{Sembedding}, we may assume that 
\begin{align*}
(u_{n}, v_{n})\rightharpoonup (u, v) &\mbox{ in }  \X_{0} \\
u_{n} \rightarrow u, \, v_{n}\rightarrow v &\mbox{ in } L^{q}_{loc}(\mathbb{R}^{N}) \\
(u_{n}, v_{n})\rightarrow (u, v) &\mbox{ a.e. in }  \R^{N}.
\end{align*}
This fact and $(Q2)$ allow us to deduce that $\J'_{0}(u, v)=0$.

Now, we assume that $u\not\equiv 0$ and $v\not\equiv 0$. Then, using $(u^{-}, v^{-})$ as test function, where $x^{-}=-\max\{-x, 0\}$, and recalling that $(x-y)(x^{-}-y^{-})\geq |x^{-}-y^{-}|^{2}$ for any $x, y\in \R$, we can see that
\begin{align*}
0=\langle \J'_{0}(u, v),(u^{-}, v^{-})\rangle&=\int_{\R^{N}} [(-\Delta)^{\frac{s}{2}} u (-\Delta)^{\frac{s}{2}}u^{-}+(-\Delta)^{\frac{s}{2}} v (-\Delta)^{\frac{s}{2}}v^{-}] dx\\
&+\int_{\R^{N}} (V_{0} u u^{-}+W_{0} v v^{-}) dx -\int_{\R^{N}} (Q_{u}(u, v)u^{-}+Q_{v}(u, v) v^{-}) dx \\
&=\iint_{\R^{2N}} \left[\frac{(u(x)-u(y))(u^{-}(x)-u^{-}(y))}{|x-y|^{N+2s}}+ \frac{(v(x)-v(y))(v^{-}(x)-v^{-}(y))}{|x-y|^{N+2s}}\right] dx dy \\
&+\int_{\R^{N}} (V_{0} u u^{-}+W_{0} v v^{-}) dx -\int_{\R^{N}} (Q_{u}(u, v)u^{-}+Q_{v}(u, v) v^{-}) dx \\
&\geq \iint_{\R^{2N}} \left[\frac{|u^{-}(x)-u^{-}(y)|^{2}}{|x-y|^{N+2s}} + \frac{|v^{-}(x)-v^{-}(y)|^{2}}{|x-y|^{N+2s}}\right] dx dy \\
& +\int_{\R^{N}} (V_{0} (u^{-})^{2}+W_{0} (v^{-})^{2}) dx = \|(u^{-}, v^{-})\|^{2}_{0},
\end{align*} 
where we used the fact that $Q_{u}=0$ on $(-\infty, 0)\times \R$ and $Q_{v}=0$ on $\R\times (-\infty, 0)$. 

Accordingly, $u, v\geq 0$ in $\R^{N}$. Now, we know that $\nabla Q$ is $(q-1)$-homogeneous, so using conditions $(Q4)$ and $(Q5)$, and applying the Mean Value Theorem, we can deduce that $Q_{u}, Q_{v}\geq 0$. In view of $(Q2)$, we can
see that $z=u+v$ is a solution to $(-\Delta)^{s}z+V_{0} z\leq C z^{q-1}$ in $\R^{N}$, for some constant $C>0$.
Hence, using a Moser iteration argument (see for instance Proposition $5.1.1$ in \cite{DMV} or Theorem $1.2$ in \cite{A3}) we can prove that $z\in L^{\infty}(\R^{N})$, which implies that $u, v\in L^{\infty}(\R^{N})$. Then $Q_{u}(u, v)$ and $Q_{v}(u, v)$ are bounded, and by applying Proposition $2.9$ in \cite{Silvestre} we have $u, v\in C^{0, \alpha}(\R^{N})\cap L^{\infty}(\R^{N})$. From the Harnack inequality \cite{CabSir}, we get $u, v>0$ in $\R^{N}$.

At this point, we can show that $\J_{0}(u, v)=c_{0}$. Indeed, taking into account $(u, v)\in \N_{0}$, \eqref{2.1} and using Fatou's Lemma, we get
\begin{align*}
c_{0}\leq \J_{0}(u, v)&=\frac{q-2}{2} \int_{\R^{N}} Q(u, v) dx \\
&\leq \liminf_{n\rightarrow \infty} \frac{q-2}{2} \int_{\R^{N}} Q(u_{n}, v_{n}) dx \\
&=\liminf_{n\rightarrow \infty} \left[ \J_{0}(u_{n}, v_{n})-\frac{1}{2} \langle \J'_{0}(u_{n}, v_{n}), (u_{n}, v_{n})\rangle\right] \\
&=c_{0}, 
\end{align*}
which yields $\J_{0}(u, v)=c_{0}$.

Secondly, we consider the case $u\equiv 0$ or $v\equiv 0$.  If $u\equiv 0$, we can use $\langle \J'_{0}(u, v),(u, v)\rangle=0$ and \eqref{2.1} to see that
$$
\|(0, v)\|^{2}_{0}=\int_{\R^{N}} Q_{u}(0, v) vdx= q\int_{\R^{N}} Q(0, v) dx=0,
$$
that is $v\equiv 0$. Analogously, we can prove that $v\equiv 0$ implies $u\equiv 0$.
Therefore, if $u\equiv 0$ or $v\equiv 0$, we have $(u, v)=(0, 0)$.\\
Since $c_{0}>0$ and $\J_{0}$ is continuous, we can deduce that $\|(u_{n}, v_{n})\|_{0}\nrightarrow 0$. Then, in view of Lemma \ref{lem2.1}, we can find a sequence $\{y_{n}\}\subset \R^{N}$ and constants $R, \gamma>0$ such that 
\begin{equation}\label{FF1}
\liminf_{n\rightarrow \infty} \int_{B_{R}(y_{n})} (|u_{n}|^{2}+|v_{n}|^{2}) dx\geq \gamma>0.
\end{equation}
Let us define $(\tilde{u}_{n}(x), \tilde{v}_{n}(x)):=(u_{n}(x+y_{n}), u_{n}(x+y_{n}))$. 
Then, using the invariance of $\R^{N}$ by translation, we can infer that $\J_{0}(\tilde{u}_{n}, \tilde{v}_{n})\rightarrow c_{0}$ and $\J'_{0}(\tilde{u}_{n}, \tilde{v}_{n})\rightarrow 0$. Since $\{(u_{n}, v_{n})\}$ is bounded in $\X_{0}$, we may assume that $(\tilde{u}_{n}, \tilde{v}_{n})\rightharpoonup (\tilde{u}, \tilde{v})$ in $\X_{0}$,  $\tilde{u}_{n}\rightarrow \tilde{u}$ and  $\tilde{v}_{n}\rightarrow \tilde{v}$ in $L^{2}_{loc}(\R^{N})$, for some $(\tilde{u}, \tilde{v})\in \X_{0}$ which is a critical point of $\J_{0}$. \\
Thus, in view of \eqref{FF1}, we have
$$
\int_{B_{R}(0)} (|\tilde{u}|^{2}+|\tilde{v}|^{2}) dx=\liminf_{n\rightarrow \infty} \int_{B_{R}(y_{n})} (|u_{n}|^{2}+|v_{n}|^{2}) dx\geq \gamma,
$$
which implies that $\tilde{u}\not\equiv 0$ or $\tilde{v}\not\equiv 0$. Arguing as before, we can obtain that both $\tilde{u}$ and $\tilde{v}$ are not identically zero.
This ends the proof of theorem.
\end{proof}

\section{compactness properties}
In this section we study the compactness properties of the functionals $\J_{\e}$. Firstly, we introduce some notation which we will use in the sequel.

If $\max\{V_{\infty}, W_{\infty}\}<\infty$, we define the functional $\J_{\infty}: \X_{0}\rightarrow \R$ by setting
$$
\J_{\infty}(u, v)=\frac{1}{2}\left(\int_{\R^{N}} |(-\Delta)^{\frac{s}{2}} u|^{2}+ |(-\Delta)^{\frac{s}{2}} v|^{2} dx+\int_{\R^{N}} (V_{\infty} u^{2}+W_{\infty} v^{2}) dx \right)-\int_{\R^{N}} Q(u, v) dx,
$$
and we denote by $c_{\infty}$ the ground state level of $\J_{\infty}$, that is
$$
c_{\infty}=\inf_{(u, v)\in \N_{\infty}} \J_{\infty}(u, v)=\inf_{(u, v)\in \X_{0}\setminus\{(0, 0)\}} \max_{t\geq 0} \J_{\infty}(t u, tv)>0,
$$
where $\N_{\infty}=\{(u, v)\in \X_{0}\setminus\{(0, 0)\}: \langle \J_{\infty}'(u, v),(u, v)\rangle=0  \}$.
If $\max\{V_{\infty}, W_{\infty}\}=\infty$, we set $c_{\infty}=\infty$.

\noindent
Now, we prove some useful lemmas which allow us to deduce a fundamental compactness result for $\J_{\e}$.

\begin{lem}\label{lem2.4}
Suppose that $\max\{V_{\infty}, W_{\infty}\}<\infty$ and let $d\in \R$. Let $\{(u_{n}, v_{n})\}\subset \X_{\e}$ be a Palais-Smale sequence for $\J_{\e}$ at the level $d$ such that $(u_{n}, v_{n})\rightharpoonup (0, 0)$ in $\X_{\e}$. If $(u_{n}, v_{n})\nrightarrow (0, 0)$ in $\X_{\e}$, then $d\geq c_{\infty}$.
\end{lem}
\begin{proof}
Let $\{t_{n}\}\subset (0, \infty)$ be a sequence such that $(t_{n}u_{n}, t_{n} v_{n})\in \N_{\infty}$. We begin by proving the following claim:\\
\textbf{Claim} $t_{0}=\limsup_{n\rightarrow \infty} t_{n}\leq 1$.
Assume by contradiction that there exists $\lambda>0$ such that 
\begin{equation}\label{2.5}
t_{n}\geq 1+\lambda \mbox{ for any } n\in \mathbb{N}.
\end{equation}
Since $\{(u_{n}, v_{n})\}$ is bounded in $\X_{\e}$, we get $\langle \J'_{\e}(u_{n}, v_{n}),(u_{n}, v_{n})\rangle\rightarrow 0$, which together with \eqref{2.1} yields
\begin{equation}\label{FF2}
\int_{\R^{N}} |(-\Delta)^{\frac{s}{2}} u_{n}|^{2}+ |(-\Delta)^{\frac{s}{2}} v_{n}|^{2} dx+\int_{\R^{N}} (V(\e x) |u_{n}|^{2}+W(\e x) |v_{n}|^{2}) dx=q\int_{\R^{N}} Q(u_{n}, v_{n}) dx+o_{n}(1).
\end{equation}
Using the fact that $(t_{n}u_{n}, t_{n} v_{n})\in \N_{\infty}$ we have
\begin{equation}\label{FF3}
t_{n}^{2}\left(\int_{\R^{N}} |(-\Delta)^{\frac{s}{2}} u_{n}|^{2}+ |(-\Delta)^{\frac{s}{2}} v_{n}|^{2} dx+\int_{\R^{N}} (V_{\infty} |u_{n}|^{2}+W_{\infty} |v_{n}|^{2}) dx\right)=qt_{n}^{q}\int_{\R^{N}} Q(u_{n}, v_{n}) dx.
\end{equation}
Putting together \eqref{FF2} and \eqref{FF3} we obtain
\begin{align}\label{2.6}
q(t_{n}^{q-2}-1)\int_{\R^{N}} Q(u_{n}, v_{n}) dx=\int_{\R^{N}} [(V_{\infty}-V(\e x)) |u_{n}|^{2}+(W_{\infty}-W(\e x)) |v_{n}|^{2}] dx+o_{n}(1).
\end{align}
Now, we can see that for any $\eta>0$ there exists $R>0$ such that 
\begin{equation}\label{2.7}
V(\e x)\geq V_{\infty}-\eta, \quad W(\e x)\geq W_{\infty}-\eta \mbox{ for any } |x|\geq R.
\end{equation}
On the other hand, in view of Theorem \ref{Sembedding}, we know that $u_{n}\rightarrow u$ and $v_{n}\rightarrow v$ in $L^{t}_{loc}(\R^{N})$ for any $t\in [1, 2^{*}_{s})$. 
Taking into account this fact, $\|(u_{n}, v_{n})\|_{\e}\leq C$, \eqref{2.6} and \eqref{2.7} we have
\begin{align}\label{2.8}
q((1+\lambda)^{q-2}-1)\int_{\R^{N}} Q(u_{n}, v_{n}) dx \leq q(t_{n}^{q-2}-1) \int_{\R^{N}} Q(u_{n}, v_{n}) dx 
\leq C'\eta+o_{n}(1).
\end{align}
Since $\|(u_{n}, v_{n})\|_{\e}\nrightarrow 0$, we can proceed as in the proof of Lemma \ref{lem2.1} to deduce that there exist a sequence $\{y_{n}\}\subset \R^{N}$ and constants $R, \gamma>0$ such that 
\begin{align}\label{2.9}
\int_{B_{R}(y_{n})} (|u_{n}|^{2}+|v_{n}|^{2}) dx\geq \gamma>0.
\end{align}
Let us define $(\tilde{u}_{n}(x), \tilde{v}_{n}(x))=(u_{n}(x+y_{n}), u_{n}(x+y_{n}))$.
Then, we may assume that $(\tilde{u}_{n}, \tilde{v}_{n})\rightharpoonup (u, v)$ in $\X_{\e}$, for some nonnegative functions $u$ and $v$ such that $\J'_{\e}(u, v)=0$.
From \eqref{2.9}, it is easy to see that $u\not\equiv 0$ or $v\not\equiv 0$. Moreover, arguing as in the proof of Theorem \ref{prop2.2}, we deduce that $u$ and $v$ are positive in $\R^{N}$.
Then, using Fatou's Lemma and \eqref{2.8} we get
$$
0<q((1+\lambda)^{q-2}-1)\int_{\R^{N}} Q(u, v) dx \leq C'\eta
$$
for any $\eta>0$, and this gives a contradiction. Therefore we can infer that $t_{0}\leq 1$. 

\noindent
Now, it is convenient to distinguish the following cases.\\
\textbf{Case 1} $t_{0}<1$. Then, we may assume that $t_{n}<1$ for all $n\in \mathbb{N}$.\\
From \eqref{2.1} we can see that
\begin{align*}
c_{\infty}\leq \J_{\infty}(t_{n}u_{n}, t_{n} v_{n})&=\J_{\infty}(t_{n}u_{n}, t_{n} v_{n})-\frac{1}{2} \langle \J'_{\infty}(t_{n}u_{n}, t_{n} v_{n}),(t_{n}u_{n}, t_{n} v_{n})\rangle \\
&=t_{n}^{q}\left(\frac{q-2}{2}\right) \int_{\R^{N}} Q(u_{n}, v_{n}) dx\\
&\leq \left(\frac{q-2}{2}\right) \int_{\R^{N}} Q(u_{n}, v_{n}) dx\\
&=\J_{\e}(t_{n}u_{n}, t_{n} v_{n})-\frac{1}{2} \langle \J'_{\e}(u_{n}, v_{n}),(u_{n}, v_{n})\rangle \\
&=d+o_{n}(1)
\end{align*}
so we deduce that $d\geq c_{\infty}$.

\noindent
\textbf{Case 2} $t_{0}=1$. Up to a subsequence, we may assume that $t_{n}\rightarrow 1$. Furthermore we have
\begin{align}\label{FF4}
d+o_{n}(1)\geq c_{\infty}+\J_{\e}(u_{n}, v_{n})-\J_{\infty}(t_{n} u_{n}, t_{n} v_{n}).
\end{align}
Now fix $\eta>0$. Taking into account \eqref{2.7}, $q$-homogeneity of $Q$, the boundedness of $\{(u_{n}, v_{n})\}$ and $t_{n}\rightarrow 1$, we can see that
\begin{align}\label{FF5}
\J_{\e}(u_{n}, v_{n})-\J_{\infty}(t_{n} u_{n}, t_{n} v_{n})&=\frac{(1-t_{n}^{2})}{2} \left(\int_{\R^{N}} |(-\Delta)^{\frac{s}{2}} u_{n}|^{2}+ |(-\Delta)^{\frac{s}{2}} v_{n}|^{2} dx\right) \nonumber\\
&+\frac{1}{2} \int_{\R^{N}} V(\e x) |u_{n}|^{2}+W(\e x) |v_{n}|^{2} dx \nonumber \\
&-\frac{t_{n}^{2}}{2} \int_{\R^{N}} (V_{\infty} |u_{n}|^{2}+W_{\infty} |v_{n}|^{2}) \, dx+(t_{n}^{q}-1) \int_{\R^{N}} Q(u_{n}, v_{n}) dx \nonumber\\
&\geq o_{n}(1)-C\eta.
\end{align}
Putting together \eqref{FF4} and \eqref{FF5}, and from the arbitrariness of $\eta$ we conclude that $d\geq c_{\infty}$.

\end{proof}

\begin{lem}\label{lem2.5}
Assume that $\max\{V_{\infty}, W_{\infty}\}=\infty$. Let $\{(u_{n}, v_{n})\}\subset \X_{\e}$ be a Palais-Smale sequence for $\J_{\e}$ at the level $d$ such that $(u_{n}, v_{n})\rightharpoonup (0, 0)$ in $\X_{\e}$. Then $(u_{n}, v_{n})\rightarrow (0, 0)$ in $\X_{\e}$.
\end{lem}
\begin{proof}
For any $(a, b)\in \R^{2}_{+}$, we define
$$
c_{(a, b)}=\inf_{(u, v)\in \X_{0}\setminus\{(0, 0)\}} \max_{t\geq 0} \J_{(a, b)} (tu, tv),
$$
where 
$$
\J_{(a, b)} (u, v)=\frac{1}{2} \left(\int_{\R^{N}} |(-\Delta)^{\frac{s}{2}} u|^{2}+ |(-\Delta)^{\frac{s}{2}} v|^{2} dx\right)+\frac{1}{2} \int_{\R^{N}} (a |u|^{2}+b|v|^{2}) dx-\int_{\R^{N}} Q(u, v) dx.
$$
We note that if $a>a'$ then $c_{(a, b)}>c_{(a', b)}$ and that $\lim_{a^{2}+b^{2}\rightarrow \infty} c_{(a, b)}=\infty$.

Now, fixed $(a, b)\in \R^{2}_{+}$, we can proceed as in the proof of Theorem \ref{prop2.2} to see that  $c_{(a, b)}$ is achieved in some couple $(u, v)$ where $u$ and $v$ are positive functions in $\R^{N}$. 

Since $\max\{V_{\infty}, W_{\infty}\}=\infty$ we can take $(a, b)\in \R^{2}_{+}$ such that $c_{(a, b)}>d$ and for any fixed $\eta>0$ there exists $R>0$ such that
\begin{equation}\label{2.10}
V(\e x)\geq a-\eta, \quad W(\e x)\geq b-\eta \mbox{ for any } |x|\geq R.
\end{equation}
We observe that if $W_{\infty}<\infty$ we can choose $b=W_{\infty}$ and $a>0$ large, and when $V_{\infty}=W_{\infty}=\infty$ we take both $a$ and $b$ sufficiently large. \\
If by contradiction $(u_{n}, v_{n})\nrightarrow (0, 0)$ in $\X_{\e}$, we argue as in the proof of Lemma \ref{lem2.4} and using \eqref{2.10} we deduce that $d\geq c_{(a, b)}$. But this is impossible because we chose $(a, b)$ such that  $c_{(a, b)}>d$. Therefore we can conclude that $(u_{n}, v_{n})\rightarrow (0, 0)$ in $\X_{\e}$.
\end{proof}

\noindent
Now, we are ready to give the proof of the following compactness result.
\begin{thm}\label{prop2.3}
The functional $\J_{\e}$ constrained to $\N_{\e}$ satisfies the Palais-Smale condition at every level $d<c_{\infty}$.
\end{thm}
\begin{proof}
Let  $\{(u_{n}, v_{n})\}\subset \N_{\e}$ be a sequence such that $\J_{\e}(u_{n}, v_{n})\rightarrow d$ and $\|\J'_{\e}(u_{n}, v_{n})\|_{*}\rightarrow 0$.
Then (see \cite{W}) there exists $\{\lambda_{n}\}\subset \R$ such that 
$$
\J'_{\e}(u_{n}, v_{n})=\lambda_{n} \mathcal{I}'_{\e}(u_{n}, v_{n})+o_{n}(1),
$$
where 
$$
\mathcal{I}_{\e}(u, v):=\|(u, v)\|^{2}_{\e}-q\int_{\R^{N}} Q(u, v) dx.
$$

Hence
\begin{align*}
0=\langle \J'_{\e}(u_{n}, v_{n}), (u_{n}, v_{n})\rangle=\lambda_{n} \langle\mathcal{I}'_{\e}(u_{n}, v_{n}),(u_{n}, v_{n})\rangle+o_{n}(1)=\lambda_{n}(2-q)\|(u_{n}, v_{n})\|^{2}_{\e}+o_{n}(1),
\end{align*}
and using \eqref{2.3} we deduce that $\lambda_{n}\rightarrow 0$. Then $\J'_{\e}(u_{n}, v_{n})\rightarrow 0$ in the dual of $\X_{\e}$.\\
Since the Palais-Smale of $\J_{\e}$ is bounded, we may assume that $(u_{n}, v_{n})\rightharpoonup (u, v)$ in $\X_{\e}$, for some $(u, v)$ which is a critical point of $\J_{\e}$.\\ 
Now, we set $(w_{n}, z_{n}):=(u_{n}-u, v_{n}-v)$. From the weak convergence of $\{(u_{n}, v_{n})\}$ and \eqref{2.2}, we can apply the Brezis-Lieb Lemma and the splitting Lemma (see for instance Lemma 4.7 in \cite{ZZ}), to deduce that
\begin{align*}
\J_{\e}(w_{n}, z_{n})&=\J_{\e}(u_{n}, v_{n})-\J_{\e}(u, v)+o_{n}(1)\\
&= d-\J_{\e}(u, v)+o_{n}(1)=:\tilde{d}+o_{n}(1)
\end{align*}
and
$$
\J'_{\e}(w_{n}, z_{n})=o_{n}(1).
$$ 
Since $\J'_{\e}(u, v)=0$, we can see that 
$$
\J_{\e}(u, v)=\J_{\e}(u, v)-\frac{1}{2}\langle \J'_{\e}(u, v),(u, v)\rangle=\frac{q-2}{2}\int_{\R^{N}} Q(u, v) dx\geq 0,
$$
which implies that $\tilde{d}<c_{\infty}$. 

Now, if we assume that $\max\{V_{\infty}, W_{\infty}\}<\infty$, by Lemma \ref{lem2.4} it follows that $(w_{n}, z_{n})\rightarrow (0, 0)$ in $\X_{\e}$, that is $(u_{n}, v_{n})\rightarrow (u, v)$ in $\X_{\e}$. In the case  $\max\{V_{\infty}, W_{\infty}\}=\infty$, we can apply  Lemma \ref{lem2.5} to deduce that $(u_{n}, v_{n})\rightarrow (u, v)$ in $\X_{\e}$.
\end{proof}

\noindent
Arguing as in the above theorem, it is easy to prove that the following result holds true. 
\begin{cor}\label{cor2.6}
The critical points of $\J_{\e}$ constrained to $\N_{\e}$ are critical points of $\J_{\e}$ in $\X_{\e}$
\end{cor}

\section{barycenter map and multiplicity of solutions to \eqref{P'}}

In this section, our main purpose is to apply the Ljusternik-Schnirelmann category theory to prove a multiplicity result for system \eqref{P'}. In order to obtain our main result, we first  give some useful lemmas.
\begin{lem}\label{lem3.1}
Let $\e_{n}\rightarrow 0$ and $\{(u_{n}, v_{n})\}\subset \N_{\e_{n}}$ be such that $\J_{\e_{n}}(u_{n}, v_{n})\rightarrow c_{0}$. Then there exists $\{\tilde{y}_{n}\}\subset \R^{N}$ such that the translated sequence 
\begin{equation*}
(\tilde{u}_{n}(x), \tilde{v}_{n}(x)):=(u_{n}(x+ \tilde{y}_{n}), v_{n}(x+\tilde{y}_{n}))
\end{equation*}
has a subsequence which converges in $\X_{0}$. Moreover, up to a subsequence, $\{y_{n}\}:=\{\e_{n}\tilde{y}_{n}\}$ is such that $y_{n}\rightarrow y\in M$. 
\end{lem}

\begin{proof}
Since $\langle \J'_{\e_{n}}(u_{n}, v_{n}), (u_{n}, v_{n}) \rangle=0$ and $\J_{\e_{n}}(u_{n}, v_{n})\rightarrow c_{0}$, we can argue as in the proof of Proposition \ref{prop2.2} to deduce that $\{(u_{n}, v_{n})\}$ is bounded. 
Let us observe that $\|(u_{n}, v_{n})\|\nrightarrow 0$ since $c_{0}>0$. Therefore, as in the proof of Lemma \ref{lem2.1}, we can find a sequence $\{\tilde{y}_{n}\}\subset \R^{N}$ and constants $R, \gamma>0$ such that
\begin{equation*}
\liminf_{n\rightarrow \infty}\int_{B_{R}(y_{n})} (|u_{n}|^{2}+|v_{n}|^{2}) dx\geq \gamma,
\end{equation*}
which implies that
\begin{equation*}
(\tilde{u}_{n}, \tilde{v}_{n})\rightharpoonup (\tilde{u}, \tilde{v}) \mbox{ weakly in } \X_{0},  
\end{equation*}
where $(\tilde{u}_{n}(x), \tilde{v}_{n}(x)):=(u_{n}(x+ \tilde{y}_{n}), v_{n}(x+\tilde{y}_{n}))$ and $(\tilde{u}, \tilde{v})\neq (0,0)$. \\
Let $\{t_{n}\}\subset (0, +\infty)$ be such that $(\hat{u}_{n}, \hat{v}_{n}):=(t_{n}\tilde{u}_{n}, t_{n}\tilde{v}_{n})\in \N_{0}$ and set $y_{n}:=\e_{n}\tilde{y}_{n}$.  \\
Using the change of variables $z\mapsto x+ \tilde{y}_{n}$ we can see that
\begin{align*}
\J_{0}(\hat{u}_{n}, \hat{v}_{n})&\leq \frac{t_{n}^{2}}{2} \left(\int_{\R^{N}} |(-\Delta)^{\frac{s}{2}} \tilde{u}_{n}|^{2}+ |(-\Delta)^{\frac{s}{2}} \tilde{v}_{n}|^{2} dx\right) - \int_{\R^{N}} Q(t_{n}\tilde{u}_{n}, t_{n}\tilde{v}_{n})\, dx \\
&+ \frac{t_{n}^{2}}{2} \int_{\R^{N}} (V(\e_{n}(x+\tilde{y}_{n}))|\tilde{u}_{n}|^{2} + W(\e_{n}(x+\tilde{y}_{n}))|\tilde{v}_{n}|^{2})\, dx\\
&=\J_{\e_{n}}(t_{n}u_{n}, t_{n}v_{n}) \leq \J_{\e_{n}}(u_{n}, v_{n})= c_{0}+ o_{n}(1). 
\end{align*}
Taking into account that $c_{0}\leq \J_{0}(\hat{u}_{n}, \hat{v}_{n})$, we can infer $\J_{0}(\hat{u}_{n}, \hat{v}_{n})\rightarrow c_{0}$. \\
Now, the sequence $\{t_{n}\}$ is bounded since $\{(\tilde{u}_{n}, \tilde{v}_{n})\}$ and $\{(\hat{u}_{n}, \hat{v}_{n})\}$ are bounded and $(\tilde{u}_{n}, \tilde{v}_{n})\nrightarrow 0$. Therefore, up to a subsequence, $t_{n}\rightarrow t_{0}\geq 0$. Indeed $t_{0}>0$. Otherwise, if $t_{0}=0$, from the boundedness of $\{(\tilde{u}_{n}, \tilde{v}_{n})\}$, we get $(\hat{u}_{n}, \hat{v}_{n})= t_{n}(\tilde{u}_{n}, \tilde{v}_{n}) \rightarrow (0,0)$, that is $\J_{0}(\hat{u}_{n}, \hat{v}_{n})\rightarrow 0$ in contrast with $c_{0}>0$. Thus $t_{0}>0$ and up to a subsequence we have $(\hat{u}_{n}, \hat{v}_{n})\rightharpoonup t_{0}(\tilde{u}, \tilde{v})= (\hat{u}, \hat{v})$ weakly in $\X_{0}$. 
Hence it holds
\begin{equation*}
\J_{0}(\hat{u}_{n}, \hat{v}_{n})\rightarrow c_{0} \quad \mbox{ and } \quad (\hat{u}_{n}, \hat{v}_{n})\rightharpoonup (\hat{u}, \hat{v}) \mbox{ weakly in } \X_{0}.
\end{equation*}
From Theorem \ref{prop2.2} we deduce that $(\hat{u}_{n}, \hat{v}_{n})\rightarrow (\hat{u}, \hat{v})$ in $\X_{0}$, that is $(\tilde{u}_{n}, \tilde{v}_{n})\rightarrow (\tilde{u}, \tilde{v})$ in $\X_{0}$. \\
Now we show that $\{y_{n}\}$ has a subsequence such that $y_{n}\rightarrow y\in M$. 
Assume by contradiction that $\{y_{n}\}$ is not bounded, that is there exists a subsequence, still denoted by $\{y_{n}\}$, such that $|y_{n}|\rightarrow +\infty$. \\
Firstly, we deal with the case $\max \{V_{\infty}, W_{\infty}\}=\infty$. \\
Since $(u_{n}, v_{n})\in \N_{\e_{n}}$ we can see that
\begin{align*}
q \int_{\R^{N}} Q(\tilde{u}_{n} ,\tilde{v}_{n})\, dx\geq \int_{\R^{N}} V(\e_{n}x+y_{n})|\tilde{u}_{n}|^{2} dx + \int_{\R^{N}} W(\e_{n}x+y_{n})|\tilde{v}_{n}|^{2} dx.
\end{align*}
Applying Fatou's Lemma, we deduce that
\begin{align}
\liminf_{n\rightarrow \infty} \int_{\R^{N}} Q(\tilde{u}_{n} ,\tilde{v}_{n})\, dx= \infty, 
\end{align}
which is impossible because the boundedness of $\{(u_{n}, v_{n})\}$ and \eqref{2.2} yield 
$$
\left|\int_{\R^{N}} Q(\tilde{u}_{n} ,\tilde{v}_{n})\, dx\right|\leq C \mbox{ for any } n\in \mathbb{N}.
$$
Let us consider the case $\max\{V_{\infty}, W_{\infty}\}<\infty$. \\
Since $(\hat{u}_{n}, \hat{v}_{n})\rightarrow (\hat{u}, \hat{v})$ strongly in $\X_{0}$ and $V_{0}<\max\{V_{\infty}, W_{\infty}\}$, we have
\begin{align}\label{3.1}
c_{0}&= \J_{0}(\hat{u}, \hat{v}) < \J_{\infty} (\hat{u}, \hat{v}) \nonumber\\
&\leq \liminf_{n\rightarrow \infty} \Bigl\{ \frac{1}{2}\left(\int_{\R^{N}} |(-\Delta)^{\frac{s}{2}} \hat{u}_{n}|^{2}+ |(-\Delta)^{\frac{s}{2}} \hat{v}_{n}|^{2} dx\right) - \int_{\R^{N}} Q(\hat{u}_{n}, \hat{v}_{n})\, dx \nonumber\\
&+ \frac{1}{2} \int_{\R^{N}} (V(\e_{n}x + y_{n})|\hat{u}_{n}|^{2} + W(\e_{n}x+y_{n})|\hat{v}_{n}|^{2})\, dx\Bigr\} \nonumber \\
&=\liminf_{n\rightarrow \infty} \J_{\e_{n}}(t_{n}u_{n}, t_{n}v_{n}) \leq \liminf_{n\rightarrow \infty} \J_{\e_{n}} (u_{n}, v_{n})=c_{0}
\end{align}
which leads to a contradiction. \\
Thus $\{y_{n}\}$ is bounded and, up to a subsequence, we may assume that $y_{n}\rightarrow y$. If $y\notin M$ then $V_{0}<\max\{V(y), W(y)\}$ and we have
$$
c_{0}=\J_{0}(\hat{u}, \hat{v})<\frac{1}{2}\left(\int_{\R^{N}} |(-\Delta)^{\frac{s}{2}} \hat{u}|^{2}+ |(-\Delta)^{\frac{s}{2}} \hat{v}|^{2} dx\right)+\frac{1}{2} \int_{\R^{N}} (V(y)|\hat{u}|^{2}+W(y) |\hat{v}|^{2}) \,dx-\int_{\R^{N}} Q(\hat{u}, \hat{v}) \,dx.
$$
Repeating the same argument developed in \eqref{3.1}, we get a contradiction. Therefore we can conclude that $y\in M$.
\end{proof}

\noindent
For any $\delta>0$ we set
$$
M_{\delta}=\{x\in \R^{N}: dist(x, M)\leq \delta\}.
$$
Let $(w_{1}, w_{2})\in \X_{0}$ be a solution for \eqref{P0} (which there exists in view of Theorem \ref{prop2.2}), and, for each $z\in M$, we define
$$
\Psi_{i, \e, z}(x)=\eta(|\e x-z|) w_{i}\left(\frac{\e x-z}{\e}\right) \quad i=1, 2.
$$
where $\eta\in C^{\infty}_{0}(\R_{+}, [0, 1])$ is a non-increasing function satisfying $\eta(t)=1$ if $0\leq t\leq \frac{\delta}{2}$ and $\eta(t)=0$ if $t\geq \delta$.\\
Let $t_{\e}>0$ be the unique positive number such that 
$$
\max_{t\geq 0} \J_{\e}(t \Psi_{1, \e, z}, t \Psi_{1, \e, z})=\J_{\e}(t_{\e} \Psi_{2,\e, z}, t_{\e} \Psi_{1,\e, z}).
$$
Finally, we consider $\Phi_{\e}(z)=(t_{\e} \Psi_{1, \e, z}, t_{\e} \Psi_{2, \e, z}) $.
Since $\J_{0}(w_{1}, w_{2})=c_{0}$ and $M$ is compact, we can prove the following result.
\begin{lem}\label{lemma3.4FS}
The functional $\Phi_{\e}$ satisfies the following limit
\begin{equation}\label{3.2}
\lim_{\e\rightarrow 0} \J_{\e}(\Phi_{\e}(y))=c_{0} \mbox{ uniformly in } y\in M.
\end{equation}
\end{lem}
\begin{proof}
Assume by contradiction that there exist $\delta_{0}>0$, $\{y_{n}\}\subset M$ and $\e_{n}\rightarrow 0$ such that 
\begin{equation}\label{4.41}
|\J_{\e_{n}}(\Phi_{\e_{n}}(y_{n}))-c_{0}|\geq \delta_{0}.
\end{equation}
We first show that $\lim_{n\rightarrow \infty}t_{\e_{n}}<\infty$.
Let us observe that using the change of variable $z=\frac{\e_{n}x-y_{n}}{\e_{n}}$, if $z\in B_{\frac{\delta}{\e_{n}}}(0)$, it follows that $\e_{n} z\in B_{\delta}(0)$ and $\e_{n} z+y_{n}\in B_{\delta}(y_{n})\subset M_{\delta}$. 

Then we have
\begin{align}\label{HeZou}
\J_{\e}(\Phi_{\e_{n}}(y_{n}))&=\frac{t_{\e_{n}}^{2}}{2}\int_{\R^{N}} |(-\Delta)^{\frac{s}{2}}(\eta(|\e_{n} z|)w_{1}(z))|^{2}\, dz+\frac{t_{\e_{n}}^{2}}{2}\int_{\R^{N}} |(-\Delta)^{\frac{s}{2}}(\eta(|\e_{n} z|)w_{2}(z))|^{2}\, dz \nonumber\\
&+\frac{t_{\e_{n}}^{2}}{2}\int_{\R^{N}} V(\e_{n} z+y_{n}) (\eta(|\e_{n} z|) w_{1}(z))^{2}\, dz+\frac{t_{\e_{n}}^{2}}{2}\int_{\R^{N}} W(\e_{n} z+y_{n}) (\eta(|\e_{n} z|) w_{2}(z))^{2}\, dz \nonumber\\
&-\int_{\R^{N}} Q(t_{\e_{n}}\eta(|\e_{n} z|)w_{1}(z), t_{\e_{n}}\eta(|\e_{n} z|)w_{2}(z)) \, dz.
\end{align}
Now let assume that $t_{\e_{n}}\rightarrow \infty$. By the definition of $t_{\e_{n}}$, $(Q1)$ and \eqref{2.1} we get
\begin{equation}\label{3.9}
\|(\Psi_{1,\e_{n}, y_{n}}, \Psi_{2,\e_{n}, y_{n}})\|^{2}_{\e_{n}}=q t_{\e_{n}}^{q-2}\int_{\R^{N}} Q(\eta(|\e_{n} z|)w_{1}(z), \eta(|\e_{n} z|)w_{2}(z)) \, dz
\end{equation}
Since $\eta=1$ in $B_{\frac{\delta}{2}}(0)$ and $B_{\frac{\delta}{2}}(0)\subset B_{\frac{\delta}{2\e_{n}}}(0)$ for $n$ big enough, and $w_{1}$, $w_{2}$ are continuous and positive in $\R^{N}$ (see proof of Theorem \ref{prop2.2}) we obtain
\begin{align}\label{3.10}
\|(\Psi_{1,\e_{n}, y_{n}}, \Psi_{2,\e_{n}, y_{n}})\|^{2}_{\e_{n}}\geq q t_{\e_{n}}^{q-2} \int_{B_{\frac{\delta}{2}}(0)} Q(w_{1}(z),w_{2}(z)) \, dz\geq C_{\delta, q} t_{\e_{n}}^{q-2},
\end{align}
where $C_{\delta, q}=q \left(\frac{\delta}{2}\right)^{N}\omega_{N}\min_{z\in \bar{B}_{\frac{\delta}{2}}(0)} Q(w_{1}(z), w_{2}(z))>0$.
Taking the limit as $n\rightarrow \infty$ in (\ref{3.10}) we can deduce that
$$
\lim_{n\rightarrow \infty} \|(\Psi_{1,\e_{n}, y_{n}}, \Psi_{2,\e_{n}, y_{n}})\|^{2}_{\e_{n}}=\infty,
$$
which is a contradiction because of
$$
\lim_{n\rightarrow \infty} \|(\Psi_{1,\e_{n}, y_{n}}, \Psi_{2,\e_{n}, y_{n}})\|^{2}_{\e_{n}}=\|(w_{1}, w_{2})\|^{2}_{0}\in (0, \infty)
$$
in view of the Dominated Convergence Theorem and Lemma 5 in \cite{PP}.\\
Thus $\{t_{\e_{n}}\}$ is bounded, and we can assume that $t_{\e_{n}}\rightarrow t_{0}\geq 0$. Clearly, if $t_{0}=0$, by limitation of $\|(\Psi_{1,\e_{n}, y_{n}}, \Psi_{2,\e_{n}, y_{n}})\|^{2}_{\e_{n}}$, the growth assumptions on $Q$, and (\ref{3.9}), we can deduce that $\|(\Psi_{1,\e_{n}, y_{n}}, \Psi_{2,\e_{n}, y_{n}})\|^{2}_{\e_{n}}\rightarrow 0$  which is impossible. Hence $t_{0}>0$.

Now, invoking the Dominated Convergence Theorem, we can see that as $n\rightarrow \infty$
$$
\int_{\R^{N}} Q(\Psi_{1, \e_{n}, y_{n}}, \Psi_{2, \e_{n}, y_{n}}) dx\rightarrow \int_{\R^{N}} Q(w_{1}, w_{2})\, dx.
$$
Then, taking the limit as $n\rightarrow \infty$ in (\ref{3.9})  we obtain
$$
\|(w_{1}, w_{2})\|^{2}_{0}=q t_{0}^{q-2} \int_{\R^{N}} Q(w_{1}, w_{2}) \, dx.
$$ 
Using the fact that $(w_{1}, w_{2})\in \mathcal{N}_{0}$ we deduce that $t_{0}=1$. Moreover, from \eqref{HeZou} we have
$$
\lim_{n\rightarrow \infty} \J_{\e}(\Phi_{\e_{n}}(y_{n}))=\J_{0}(w_{1}, w_{2})=c_{0},
$$
which is impossible thanks to (\ref{4.41}).
\end{proof}

\noindent
Now we are in the position to define the barycenter map. We take $\rho>0$ such that $M_{\delta}\subset B_{\rho}$, and we consider $\varUpsilon: \R^{N}\rightarrow \R^{N}$ given by 
 \begin{equation*}
 \varUpsilon(x)=
 \left\{
 \begin{array}{ll}
 x &\mbox{ if } |x|<\rho \\
 \frac{\rho x}{|x|} &\mbox{ if } |x|\geq \rho.
  \end{array}
 \right.
 \end{equation*}
We define the barycenter map $\beta_{\e}: \N_{\e}\rightarrow \R^{N}$ as 
\begin{align*}
\beta_{\e}(u, v)=\frac{\int_{\R^{N}} \varUpsilon(\e x)(u^{2}(x)+v^{2}(x)) dx}{\int_{\R^{N}} u^{2}(x)+v^{2}(x) dx}.
\end{align*}

\begin{lem}\label{lemma3.5FS}
The functional $\Phi_{\e}$ satisfies the following limit
\begin{equation}\label{3.3}
\lim_{\e \rightarrow 0} \beta_{\e}(\Phi_{\e}(y))=y \mbox{ uniformly in } y\in M.
\end{equation}
\end{lem}
\begin{proof}
Suppose by contradiction that there exist $\delta_{0}>0$, $\{y_{n}\}\subset M$ and $\e_{n}\rightarrow 0$ such that 
\begin{equation}\label{4.4}
|\beta_{\e_{n}}(\Phi_{\e_{n}}(y_{n}))-y_{n}|\geq \delta_{0}.
\end{equation}
Using the definitions of $\Phi_{\e_{n}}(y_{n})$, $\beta_{\e_{n}}$, $\eta$ and the change of variable $z=\frac{\e_{n} x-y_{n}}{\e_{n}}$, we can see that 
$$
\beta_{\e_{n}}(\Phi_{\e_{n}}(y_{n}))=y_{n}+\frac{\int_{\R^{N}}[\Upsilon(\e_{n}z+y_{n})-y_{n}] |\eta(|\e_{n}z|)|^{2} (|w_{1}(z)|^{2}+|w_{2}(z)|^{2}) \, dx}{\int_{\R^{N}} |\eta(|\e_{n}z|)|^{2} (|w_{1}(z)|^{2}+|w_{2}(z)|^{2})\, dx}.
$$
Taking into account $\{y_{n}\}\subset M\subset B_{\rho}$ and the Dominated Convergence Theorem we can infer that 
$$
|\beta_{\e_{n}}(\Phi_{\e_{n}}(y_{n}))-y_{n}|=o_{n}(1)
$$
which contradicts (\ref{4.4}).
\end{proof}

\noindent
At this point, we introduce a subset $\widetilde{\N}_{\e}$ of $\N_{\e}$ by taking a function $h:\R_{+}\rightarrow \R_{+}$ such that $h(\e)\rightarrow 0$ as $\e \rightarrow 0$, and setting
$$
\widetilde{\N}_{\e}=\{(u, v)\in \N_{\e}: \J_{\e}(u)\leq c_{0}+h(\e)\}.
$$
Fixed $y\in M$, we conclude from Lemma \ref{lemma3.4FS} that $h(\e)=|\J_{\e}(\Phi_{\e}(y))-c_{0}|\rightarrow 0$ as $\e \rightarrow 0$. Hence $\Phi_{\e}(y)\in \widetilde{\N}_{\e}$, and $\widetilde{\N}_{\e}\neq \emptyset$ for any $\e>0$. Moreover, we have the following lemma.
\begin{lem}\label{lemma3.7FS}
$$
\lim_{\e \rightarrow 0} \sup_{(u, v)\in \widetilde{\mathcal{N}}_{\e}} dist(\beta_{\e}(u, v), M_{\delta})=0.
$$
\end{lem}

\begin{proof}
Let $\e_{n}\rightarrow 0$ as $n\rightarrow \infty$. For any $n\in \mathbb{N}$ there exists $(u_{n}, v_{n})\in \widetilde{\N}_{\e_{n}}$ such that
$$
\sup_{(u, v)\in \widetilde{\N}_{\e_{n}}} \inf_{y\in M_{\delta}}|\beta_{\e_{n}}(u, v)-y|=\inf_{y\in M_{\delta}}|\beta_{\e_{n}}(u_{n}, v_{n})-y|+o_{n}(1).
$$
Therefore it is suffices to prove that there exists $\{y_{n}\}\subset M_{\delta}$ such that 
\begin{equation}\label{3.13}
\lim_{n\rightarrow \infty} |\beta_{\e}(u_{n}, v_{n})-y_{n}|=0.
\end{equation}
We note that $\{(u_{n}, v_{n})\}\subset  \widetilde{\N}_{\e_{n}}\subset  \N_{\e_{n}}$ from which we deduce that
$$
c_{0}\leq c_{\e_{n}}\leq \J_{\e_{n}}(u_{n}, v_{n})\leq c_{0}+h(\e_{n}).
$$
This yields $\J_{\e_{n}}(u_{n}, v_{n})\rightarrow c_{0}$. By Lemma \ref{lem3.1} there exists $\{\tilde{y}_{n}\}\subset \R^{N}$ such that $y_{n}=\e_{n}\tilde{y}_{n}\in M_{\delta}$ for $n$ sufficiently large. By setting $(\tilde{u}_{n}(x), \tilde{v}_{n}(x))=(u_{n}(\cdot+\tilde{y}_{n}), v_{n}(\cdot+\tilde{y}_{n}))$ we can see that
$$
\beta_{\e_{n}}(u_{n}, v_{n})=y_{n}+\frac{\int_{\R^{N}}[\Upsilon(\e_{n}z+y_{n})-y_{n}] (\tilde{u}_{n}^{2}+\tilde{v}_{n}^{2}) \, dz}{\int_{\R^{N}} (\tilde{u}_{n}^{2}+\tilde{v}_{n}^{2})\, dz}.
$$
Since $(\tilde{u}_{n}, \tilde{v}_{n})\rightarrow (u, v)$ in $\X_{0}$ and $\e_{n}z+y_{n}\rightarrow y\in M$, we deduce that $\beta_{\e_{n}}(u_{n}, v_{n})=y_{n}+o_{n}(1)$ that is (\ref{3.13}) holds.
\end{proof}

\noindent
Now, we are ready to provide the proof of the first multiplicity result related to \eqref{P}.
\begin{proof}[Proof of thm \ref{thm1}]
Given $\delta>0$ we can apply Lemma \ref{lemma3.4FS}, Lemma \ref{lemma3.5FS} and Lemma \ref{lemma3.7FS} to find some $\e_{\delta}>0$ such that for any $\e\in (0, \e_{\delta})$, the diagram
$$
M \stackrel{\Phi_{\e}}{\rightarrow}  \widetilde{\N}_{\e} \stackrel{\beta_{\e}}{\rightarrow} M_{\delta}
$$
is well-defined and $\beta_{\e}\circ \Phi_{\e}$ is  homotopically equivalent to the embedding $\iota: M\rightarrow M_{\delta}$. By the definition of $\widetilde{\N}_{\e}$ and taking $\e_{\delta}$ sufficiently small, we may assume that $\J_{\e}$ satisfies the Palais-Smale condition in $\widetilde{\N}_{\e}$. Therefore, standard Ljusternik-Schnirelmann theory \cite{W} provides at least $cat_{\widetilde{\N}_{\e}}(\widetilde{\N}_{\e})$ critical points $(u_{i}, v_{i})$ of $\J_{\e}$ restricted to $\N_{\e}$. Using the arguments in \cite{BC} we can see that $cat_{\widetilde{\N}_{\e}}(\widetilde{\N}_{\e})\geq cat_{M_{\delta}}(M)$. From Corollary \ref{cor2.6} and the arguments contained in the proof of Theorem \ref{prop2.2} we can conclude that $u_{i}>0$, $v_{i}>0$ and $(u_{i}, v_{i})$ is a solution to \eqref{P'}.
\end{proof}

\section{Proof of Theorem \ref{thm2}}
In this last section we deal with the nonlocal system in the critical case. As in the Section $3$, we consider the following autonomous critical system 
\begin{equation}\label{CP0}
\left\{
\begin{array}{ll}
 (-\Delta)^{s}u+V_{0}u=Q_{u}(u, v)+\frac{2\alpha}{\alpha+\beta}|u|^{\alpha-2}u |v|^{\beta}  &\mbox{ in } \R^{N}\\
 (-\Delta)^{s}v+W_{0}v=Q_{v}(u, v)+\frac{2\beta}{\alpha+\beta} |u|^{\alpha}|v|^{\beta-2}v &\mbox{ in } \R^{N} \\
u, v>0 &\mbox{ in } \R^{N},
\end{array}
\right.
\end{equation}
and define the energy functional
$$
\J_{0}(u, v)=\frac{1}{2}\|(u, v)\|^{2}_{0}-\int_{\R^{N}} Q(u, v) dx-\frac{2}{\alpha+\beta}\int_{\R^{N}} (u^{+})^{\alpha}(v^{+})^{\beta} dx,
$$
and its ground state level
$$
m_{0}=\inf_{(u, v)\in \N_{0}} \J_{0}(u, v)=\inf_{(u, v)\in X_{0}\setminus \{(0, 0)\}} \max_{t\geq 0} \J_{0}(t u, tv)>0.
$$
Now, we denote by 
\begin{equation}\label{AFS}
\widetilde{S}_{*}=\widetilde{S}_{*}(\alpha, \beta)=\inf_{u, v\in H^{s}(\R^{N})\setminus\{(0, 0)\}} \frac{\int_{\R^{N}} |(-\Delta)^{\frac{s}{2}} u|^{2}+|(-\Delta)^{\frac{s}{2}} v|^{2} dx}{\left(\int_{\R^{N}} |u|^{\alpha}|v|^{\beta} dx\right)^{\frac{2}{2^{*}_{s}}}}.
\end{equation}
In the next lemma, we prove an interesting relation between $S_{*}$ and $\widetilde{S}_{*}$.
\begin{lem}\label{lem4.1}
It holds
$$
\widetilde{S}_{*}=S_{*} \left[\left(\frac{\alpha}{\beta}\right)^{\frac{\beta}{2^{*}_{s}}}+\left(\frac{\beta}{\alpha}\right)^{\frac{\alpha}{2^{*}_{s}}}  \right].
$$
Moreover, if $w$ realizes $S_{*}$, then $(A w, B w)$ realizes $\widetilde{S}_{*}$ where $A$ and $B$ are such that $\frac{A}{B}=\sqrt{\frac{\alpha}{\beta}}$.
\end{lem}
\begin{proof}
Let $\{w_{n}\}$ be a minimizing sequence for $S_{*}$. Let $p$ and $q$ two positive numbers which will be chosen later. Taking $u_{n}= p w_{n}$ and $v_{n}=qw_{n}$ in the quotient \eqref{AFS}, we have
\begin{equation}\label{36AFS}
\frac{p^2+q^2}{(p^{\alpha} q^{\beta})^{\frac{2}{2^{*}_{s}}}} \frac{\int_{\R^{N}} |(-\Delta)^{\frac{s}{2}} w_{n}|^{2} dx}{\left(\int_{\R^{N}} |w_{n}|^{2^{*}_{s}} dx\right)^{\frac{2}{2^{*}_{s}}}}\geq \widetilde{S}_{*}.
\end{equation}
We note that
\begin{equation}\label{37AFS}
\frac{p^2+q^2}{(p^{\alpha} q^{\beta})^{\frac{2}{2^{*}_{s}}}}=\left(\frac{p}{q} \right)^{\frac{2\beta}{2^{*}_{s}}}+\left(\frac{p}{q} \right)^{-\frac{2\alpha}{2^{*}_{s}}},
\end{equation}
and we consider the function $g: \R_{+}\rightarrow \R$ defined as
$$
g(t)=t^{\frac{2\beta}{2^{*}_{s}}}+t^{-\frac{2\alpha}{2^{*}_{s}}}. 
$$ 
Then it is easy to verify that $g$ achieves its minimum at the point $t=\sqrt{\frac{\alpha}{\beta}}$ and in particular
\begin{equation}\label{38AFS}
g\left(\sqrt{\frac{\alpha}{\beta}}\right)=\left(\frac{\alpha}{\beta}\right)^{\frac{\beta}{2^{*}_{s}}}+\left(\frac{\beta}{\alpha}\right)^{\frac{\alpha}{2^{*}_{s}}}.
\end{equation}
Taking $p$ and $q$ in \eqref{36AFS} such that $\frac{p}{q}=\sqrt{\frac{\alpha}{\beta}}$ we get
$$
\left[\left(\frac{\alpha}{\beta}\right)^{\frac{\beta}{2^{*}_{s}}}+\left(\frac{\beta}{\alpha}\right)^{\frac{\alpha}{2^{*}_{s}}}  \right] \frac{\int_{\R^{N}} |(-\Delta)^{\frac{s}{2}} w_{n}|^{2} dx}{\left(\int_{\R^{N}} |w_{n}|^{2^{*}_{s}} dx\right)^{\frac{2}{2^{*}_{s}}}}\geq \widetilde{S}_{*}
$$
which gives
\begin{equation}\label{39AFS}
\left[\left(\frac{\alpha}{\beta}\right)^{\frac{\beta}{2^{*}_{s}}}+\left(\frac{\beta}{\alpha}\right)^{\frac{\alpha}{2^{*}_{s}}}  \right]  S_{*}  \geq \widetilde{S}_{*}.
\end{equation}
Now, in order to conclude the proof, we consider a minimizing sequence $\{(u_{n}, v_{n})\}$ for $\widetilde{S}_{*}$.
Let us define $z_{n}=p_{n} v_{n}$, where $p_{n}>0$ is such that 
\begin{equation}\label{40AFS}
\int_{\R^{N}} |u_{n}|^{2^{*}_{s}} dx=\int_{\R^{N}} |z_{n}|^{2^{*}_{s}} dx.
\end{equation}
Using Young's inequality and \eqref{40AFS} we can see that
\begin{align}\label{41AFS}
\int_{\R^{N}} |u_{n}|^{\alpha} |z_{n}|^{\beta} dx&\leq \frac{\alpha}{2^{*}_{s}} \int_{\R^{N}} |u_{n}|^{\alpha+\beta}  dx+\frac{\beta}{2^{*}_{s}} \int_{\R^{N}} |z_{n}|^{\alpha+\beta}  dx \nonumber\\
&= \int_{\R^{N}} |u_{n}|^{2^{*}_{s}}  dx=\int_{\R^{N}} |z_{n}|^{2^{*}_{s}}  dx.
\end{align}
Therefore, by \eqref{38AFS}, \eqref{41AFS} and $\alpha+\beta=2^{*}_{s}$ we can deduce that
\begin{align*}
\frac{\int_{\R^{N}} |(-\Delta)^{\frac{s}{2}} u_{n}|^{2}+|(-\Delta)^{\frac{s}{2}} v_{n}|^{2} dx}{\left(\int_{\R^{N}} |u_{n}|^{\alpha}|v_{n}|^{\beta} dx\right)^{\frac{2}{2^{*}_{s}}}}&=\frac{p_{n}^{\frac{2\beta}{2^{*}_{s}}}\int_{\R^{N}} |(-\Delta)^{\frac{s}{2}} u_{n}|^{2}+|(-\Delta)^{\frac{s}{2}} v_{n}|^{2} dx}{\left(\int_{\R^{N}} |u_{n}|^{\alpha}|z_{n}|^{\beta} dx\right)^{\frac{2}{2^{*}_{s}}}} \\
&\geq p_{n}^{\frac{2\beta}{2^{*}_{s}}} \frac{\int_{\R^{N}} |(-\Delta)^{\frac{s}{2}} u_{n}|^{2}dx}{\left(\int_{\R^{N}} |u_{n}|^{2^{*}_{s}} dx\right)^{\frac{2}{2^{*}_{s}}}}+p_{n}^{\frac{2\beta}{2^{*}_{s}}} p_{n}^{-2} \frac{\int_{\R^{N}} |(-\Delta)^{\frac{s}{2}} z_{n}|^{2}dx}{\left(\int_{\R^{N}} |z_{n}|^{2^{*}_{s}} dx\right)^{\frac{2}{2^{*}_{s}}}} \\
&\geq S_{*} \left(p_{n}^{\frac{2\beta}{2^{*}_{s}}}+p_{n}^{\frac{2\beta}{2^{*}_{s}}-2}\right)=S_{*} g(p_{n}) \\
&\geq S_{*} g\left(\sqrt{\frac{\alpha}{\beta}}\right)= S_{*} \left[\left(\frac{\alpha}{\beta}\right)^{\frac{\beta}{2^{*}_{s}}}+\left(\frac{\beta}{\alpha}\right)^{\frac{\alpha}{2^{*}_{s}}}\right].
\end{align*}
The end of the proof is obtained by passing to the limit in the above inequality.
\end{proof}

\noindent
Next, we prove the "critical version" of Lemma \ref{lem2.1}.
\begin{lem}\label{lem4.2}
Let $\{(u_{n}, v_{n})\}\subset \X_{0}$ be a Palais-Smale sequence for $\J_{0}$ at the level $d<\frac{2s}{N} \left(\frac{\tilde{S}_{*}}{2}\right)^{\frac{N}{2s}}$ and $(u_{n}, v_{n})\rightharpoonup (0,0)$. Then, one of the following conclusions holds:
\begin{compactenum}[(i)]
\item $\|(u_{n}, v_{n})\|_{0}\rightarrow 0$, or
\item there exist a sequence $\{y_{n}\}\subset \R^{N}$ and constants $R, \gamma>0$ such that 
$$
\liminf_{n\rightarrow \infty}\int_{B_{R}(y_{n})} (|u_{n}|^{2}+|v_{n}|^{2}) dx\geq \gamma.
$$
\end{compactenum}
\end{lem}
\begin{proof}
Assume that $(ii)$ does not hold. Then, for any $R>0$, we get
$$
\lim_{n\rightarrow \infty} \sup_{y\in \R^{N}}\int_{B_{R}(y)} |u_{n}|^{2} dx=0=\lim_{n\rightarrow \infty} \sup_{y\in \R^{N}}\int_{B_{R}(y)} |v_{n}|^{2} dx.
$$
Using lemma \ref{lionslemma} it follows that 
$$
u_{n}, v_{n}\rightarrow 0 \mbox{ in } L^{r}(\R^{N}) \quad \forall r\in (2, 2^{*}_{s}),
$$
and in view of \eqref{2.2} we can see that $\int_{\R^{N}} Q(u_{n}, v_{n}) dx \rightarrow 0$.\\
Since $\{(u_{n}, v_{n})\}$ is bounded we have $\langle \J'_{0}(u_{n}, v_{n}),(u_{n}, v_{n})\rangle\rightarrow 0$.
Then we obtain
$$
\|(u_{n}, v_{n})\|^{2}_{0}-2 \int_{\R^{N}} (u_{n}^{+})^{\alpha} (v_{n}^{+})^{\beta} dx=o_{n}(1),
$$
which implies that there exists $L\geq 0$ such that
\begin{equation}\label{4.1}
\|(u_{n}, v_{n})\|^{2}_{0}\rightarrow L \mbox{ and }  \int_{\R^{N}} (u_{n}^{+})^{\alpha} (v_{n}^{+})^{\beta} dx\rightarrow \frac{L}{2}.
\end{equation}
Since $\J_{0}(u_{n}, v_{n})\rightarrow d$ we can use \eqref{4.1} to deduce that $d=\frac{Ls}{N}$. By the definition of $\widetilde{S}_{*}$ we get
$$
\|(u_{n}, v_{n})\|^{2}_{0}\geq \widetilde{S}_{*} \left(\int_{\R^{N}} |u_{n}|^{\alpha}|v_{n}|^{\beta} dx\right)^{\frac{2}{2^{*}_{s}}}\geq \widetilde{S}_{*} \left(\int_{\R^{N}} (u_{n}^{+})^{\alpha}(v_{n}^{+})^{\beta} dx\right)^{\frac{2}{2^{*}_{s}}},
$$
which gives $L\geq \widetilde{S}_{*} (\frac{L}{2})^{\frac{2}{2^{*}_{s}}}$. Now, if $L>0$ we obtain $Nd=sL\geq 2s \left(\frac{\tilde{S}_{*}}{2}\right)^{\frac{N}{2s}}$ which provides a contradiction. Thus $L=0$ and $(i)$ holds true.
\end{proof}

\noindent
Now we prove that the critical autonomous system admits a nontrivial solution.
\begin{thm}\label{prop4.3}
The problem \eqref{CP0} has a weak solution.
\end{thm}
\begin{proof}
Since $\J_{0}$ has a Mountain Pass geometry, there exists $\{(u_{n}, v_{n})\}\subset \X_{0}$ such that 
$$
\J_{0}(u_{n}, v_{n})\rightarrow m_{0} \mbox{ and }  \J'_{0}(u_{n}, v_{n})\rightarrow 0.
$$
We aim to show that 
\begin{equation}\label{FFc}
m_{0}<\frac{2s}{N} \left(\frac{\tilde{S}_{*}}{2}\right)^{\frac{N}{2s}}.
\end{equation}
Indeed, once proved \eqref{FFc}, we can repeat the same arguments developed in the proof of Theorem \ref{prop2.2} and  applying Lemma \ref{lem4.2} instead of Lemma \ref{lem2.1}, we deduce the existence of a weak solution to \eqref{CP0}.
By the definition of $m_{0}$ it is enough to prove that there exists $(u, v)\in \X_{0}$ such that 
$$
\max_{t\geq 0} \J_{\e}(tu, tv)<\frac{2s}{N} \left(\frac{\tilde{S}_{*}}{2}\right)^{\frac{N}{2s}}.
$$
Let $A, B>0$ such that $\frac{A}{B}=\sqrt{\frac{\alpha}{\beta}}$. Then, in view of Lemma \ref{lem4.1} we  can deduce that 
$$
\widetilde{S}_{*}=S_{*} \frac{(A^2+B^2)}{(A^{\alpha} B^{\beta})^{\frac{2}{2^{*}_{s}}}}.
$$
Fix $\eta \in C^{\infty}_{0}(\R^{N})$ a cut-off function such that $0\leq \eta \leq 1$, $\eta=1$ on $B_{r}$ and $\eta=0$ on $\R^{N}\setminus B_{2r}$, where $B_{r}$ denotes the ball in $\R^{N}$ of center at origin and radius $r$.\\
For $\varepsilon>0$ let us define $v_{\varepsilon}(x)=\eta(x)z_{\varepsilon}(x)$, where
$$
z_{\varepsilon}(x)=\frac{\kappa \varepsilon^{\frac{N-2s}{2}}}{(\varepsilon^{2}+|x|^{2})^{\frac{N-2s}{2}}}
$$
is a solution to
$$
(-\Delta)^{s}u=S_{*}|u|^{2^{*}_{s}-2}u \mbox{ in } \R^{N},
$$
and $\kappa$ is a suitable positive constant depending only on $N$ and $s$.\\
Now we set
$$
u_{\varepsilon}=\frac{z_{\varepsilon}}{\left(\int_{\R^{N}} |z_{\varepsilon}|^{2^{*}_{s}} dx \right)^{\frac{1}{2^{*}_{s}}}}.
$$
By performing similar calculations to those in \cite{SV} (see Propositions $21$ and $22$), we can see that
\begin{align}\label{BN1}
\int_{\R^{N}} |(-\Delta)^{\frac{s}{2}} u_{\e}|^{2} dx\leq S_{*}+O(\varepsilon^{N-2s}),
\end{align}
\begin{equation}\label{BN2}
\int_{\R^{N}} |u_{\varepsilon}|^{2} dx=
\left\{
\begin{array}{ll}
O(\varepsilon^{2s})  &\mbox{ if } N>4s \\
O(\varepsilon^{2s} |\log(\varepsilon)|) &\mbox{ if } N=4s \\
O(\varepsilon^{N-2s}) &\mbox{ if } N<4s,
\end{array}
\right.
\end{equation}
and
\begin{equation}\label{BN3}
\int_{\R^{N}} |u_{\varepsilon}|^{q} dx=
\left\{
\begin{array}{ll}
O(\varepsilon^{\frac{2N-(N-2s)q}{2}})  &\mbox{ if } q>\frac{N}{N-2s} \\
O(|\log(\varepsilon)|\varepsilon^{\frac{N}{2}}) &\mbox{ if } q=\frac{N}{N-2s} \\
O(\varepsilon^{\frac{(N-2s)q}{2}}) &\mbox{ if } q<\frac{N}{N-2s}.
\end{array}
\right.
\end{equation}
Thus, by $(Q6)$, we can note that
\begin{align*}
\J_{0}(t A u_{\e}, t  B u_{\e})&\leq \left[\frac{t^{2}}{2} (A^{2}+B^{2}) D_{\e}-\frac{2t^{2^{*}_{s}}}{2^{*}_{s}} A^{\alpha} B^{\beta}\right]- \lambda t^{q_{1}} A^{q_{1}} B^{q_{1}} \int_{\R^{N}} |u_{\e}|^{q_{1}} dx 
\end{align*}
where $h_{\e}(t):=\frac{t^{2}}{2} (A^{2}+ B^{2}) D_{\varepsilon} - \frac{2 t^{2^{*}_{s}}}{2^{*}_{s}} A^{\alpha} B^{\beta}$, and 
$$
D_{\e}=\int_{\R^{N}} |(-\Delta)^{\frac{s}{2}} u_{\e}|^{2} dx+\int_{\R^{N}} \max\{V_{0}, W_{0}\} u_{\e}^{2} dx.
$$
Let us denote by $t_{\e}>0$ be the maximum point of $h_{\e}(t)$. Since $h'_{\e}(t_{\e})=0$ we have
$$
\bar{t}_{\e}=\left(\frac{D_{\e} (A^2+B^2)}{2(A^{\alpha} B^{\beta} )^{2/ 2^{*}_{s}}}\right)^{\frac{N-2s}{4s}}\geq t_{\e}>0.
$$
Using the fact that $h_{\e}(t)$ is increasing in $(0, \bar{t}_{\e})$, we can see that
$$
\J_{0}(t A u_{\e}, t  B u_{\e})\leq \frac{2s}{N}\left(\frac{D_{\e} (A^2+B^2)}{2(A^{\alpha} B^{\beta} )^{2/ 2^{*}_{s}}}\right)^{\frac{N}{2s}} -\lambda t^{q_{1}} A^{q_{1}} B^{q_{1}} \int_{\R^{N}} |u_{\e}|^{q_{1}} dx.
$$
Now, recalling that $(a+b)^{r}\leq a^{r}+r(a+b)^{r-1} b$ for any $a, b>0$ and $r\geq 1$, we can obtain that
$$
D_{\e}^{N/2s}\leq S_{*}^{N/2s}+O(\e^{N-2s})+ C_{1} \int_{\R^{N}} |u_{\e}|^{2} dx,
$$
On the other hand $h'_{\e}(t_{\e})=0$ and the Mountain Pass geometry of $\J_{\e}$ imply that there exists $\sigma>0$ such that
$$
t_{\e}\geq \sigma \mbox{ for any } \e>0, 
$$
that is $t_{\e}$ can be estimated from below by a constant independent of $\e$.\\ 
Then we have
\begin{align*}
\J_{0}(t A u_{\e}, t  B u_{\e})\leq \frac{2s}{N} \left(\frac{\tilde{S}_{*}}{2}\right)^{\frac{N}{2s}}+O(\e^{N-2s})+ C_{2} \int_{\R^{N}} |u_{\e}|^{2} dx -\lambda C_{3}\int_{\R^{N}} |u_{\e}|^{q_{1}} dx,
\end{align*}
where $C_{2}, C_{3}>0$ are independent of $\e$ and $\lambda$.\\
Now we distinguish the following cases: \\
If $N>4s$ then $q_{1}>\frac{N}{N-2s}$. Hence, by (\ref{BN2}) and (\ref{BN3}), we can see that 
\begin{align*}
\sup_{t\geq 0} h_{\e}(t)&\leq \frac{2s}{N} \left(\frac{\tilde{S}_{*}}{2}\right)^{\frac{N}{2s}}+O(\varepsilon^{N-2s})+O(\varepsilon^{2s})-\lambda O(\varepsilon^{\frac{2N-(N-2s)q_{1}}{2}}).
\end{align*}
Taking into account $\frac{2N-(N-2s)q_{1}}{2}<2s<N-2s$ we get the thesis for $\varepsilon$ small enough. \\
When $N=4s$ then $q_{1}\in (2, 4)$ and in particular $q_{1}>\frac{N}{N-2s}=2$, so from (\ref{BN2}) and (\ref{BN3}) we deduce that
\begin{align*}
\sup_{t\geq 0} h_{\e}(t)&\leq\frac{2s}{N} \left(\frac{\tilde{S}_{*}}{2}\right)^{\frac{N}{2s}}+O(\varepsilon^{2s})+O(\varepsilon^{2s}|\log(\varepsilon)|)-\lambda O(\varepsilon^{4s-sq_{1}})
\end{align*}
which implies (\ref{2.3}) because of $\lim_{\varepsilon\rightarrow 0} \frac{\varepsilon^{4s-sq}}{\varepsilon^{2s}(1+|\log(\varepsilon)|)}=\infty$. \\
If $2s<N<4s$ and $q_{1}\in (\frac{4s}{N-2s}, 2^{*}_{s})$ then  $q_{1}>\frac{N}{N-2s}$. Therefore we have
\begin{align*}
\sup_{t\geq 0} h_{\e}(t)&\leq \frac{2s}{N} \left(\frac{\tilde{S}_{*}}{2}\right)^{\frac{N}{2s}}+O(\varepsilon^{N-2s})+O(\varepsilon^{N-2s})-\lambda O(\varepsilon^{\frac{2N-(N-2s)q_{1}}{2}})
\end{align*}
and we obtain the conclusion for $\e$ sufficiently small in light of $\frac{2N-(N-2s)q_{1}}{2}<N-2s$. \\
If $2s<N<4s$ and $q_{1}\in (2, \frac{4s}{N-2s}]$, we argue as before and using (\ref{BN3}) we get
\begin{equation*}
\sup_{t\geq 0} h_{\e}(t)\leq 
\left\{
\begin{array}{ll}
\frac{2s}{N} \left(\frac{\tilde{S}_{*}}{2}\right)^{\frac{N}{2s}}+O(\varepsilon^{N-2s})-\lambda O(\varepsilon^{\frac{2N-(N-2s)q_{1}}{2}})  &\mbox{ if } q_{1}>\frac{N}{N-2s} \\
\frac{2s}{N} \left(\frac{\tilde{S}_{*}}{2}\right)^{\frac{N}{2s}}+O(\varepsilon^{N-2s})-\lambda O(|\log(\varepsilon)|\varepsilon^{\frac{N}{2}}) &\mbox{ if } q_{1}=\frac{N}{N-2s} \\
\frac{2s}{N} \left(\frac{\tilde{S}_{*}}{2}\right)^{\frac{N}{2s}}+O(\varepsilon^{N-2s})-\lambda O(\varepsilon^{\frac{(N-2s)q_{1}}{2}}) &\mbox{ if } q_{1}<\frac{N}{N-2s}. 
\end{array}
\right.
\end{equation*}
Then we can find $\lambda_{0}>0$ large enough such that for any $\lambda\geq \lambda_{0}$  and $\e>0$ small it holds
\begin{align*}
\sup_{t\geq 0} h_{\e}(t)< \frac{2s}{N} \left(\frac{\tilde{S}_{*}}{2}\right)^{\frac{N}{2s}}.
\end{align*}
Putting together the above estimates we can infer that for any $\e>0$ sufficiently small 
\begin{align*}
\max_{t\geq 0} \J_{0}(t A u_{\e}, t  B u_{\e})\leq \max_{t\geq 0} h_{\e}(t)=h_{\e}(t_{\e})<\frac{2s}{N} \left(\frac{\tilde{S}_{*}}{2}\right)^{\frac{N}{2s}}.
\end{align*}
\end{proof}

\noindent
Since we are interested in weak solutions of \eqref{CP}, we consider the re-scaled system
\begin{equation}\label{CP2}
\left\{
\begin{array}{ll}
(-\Delta)^{s}u+V(\e x)u=Q_{u}(u, v)+\frac{2\alpha}{\alpha+\beta} |u|^{\alpha-2}u |v|^{\beta} &\mbox{ in } \R^{N}\\
(-\Delta)^{s}v+W(\e x)v=Q_{v}(u, v)+\frac{2\beta}{\alpha+\beta} |u|^{\alpha} |v|^{\beta-2}v &\mbox{ in } \R^{N} \\
u, v>0 &\mbox{ in } \R^{N}.
\end{array}
\right.
\end{equation}
Thus, the corresponding functional $\J_{\e}: \X_{\e}\rightarrow \R$ is given by
$$
\J_{\e}(u, v)=\frac{1}{2}\|(u, v)\|^{2}_{\e}-\int_{\R^{N}} Q(u, v) dx - \frac{2}{\alpha+\beta}\int_{\R^{N}} (u^{+})^{\alpha} (v^{+})^{\beta} dx.  
$$
Clearly, the critical points of $\J_{\e}$ belong to the Nehari manifold 
\begin{equation*}
\M_{\e}:=\{(u,v)\in \X_{\e} \setminus \{(0,0)\} : \langle \J_{\e}'(u,v), (u,v)\rangle =0 \},
\end{equation*}
and the ground state level is given by
\begin{equation*}
m_{\e}:=\inf_{(u,v)\in \M_{\e}} \J_{\e}(u,v) = \inf_{(u,v)\in \X_{\e}\setminus \{(0,0)\}} \max_{t\geq 0} \J_{\e}(tu, tv)>0. 
\end{equation*}
As made in the previous sections, the Palais-Smale condition for the functional $\J_{\e}$ is related to $V_{\infty}$ and $W_{\infty}$. Then, as in Section $4$, when $\max\{V_{\infty}, W_{\infty}\}<\infty$, we define the limit functional $\J_{\infty}: \X_{0}\rightarrow \R$ by setting
\begin{align*}
\J_{\infty}(u, v):= &\frac{1}{2} \left(\int_{\R^{N}} |(-\Delta)^{\frac{s}{2}} u|^{2}+ |(-\Delta)^{\frac{s}{2}} v|^{2} dx+\int_{\R^{N}} (V_{\infty} u^{2}+W_{\infty} v^{2}) dx\right) \\
&- \int_{\R^{N}} Q(u,v)\, dx - \frac{2}{\alpha+\beta}\int_{\R^{N}} (u^{+})^{\alpha} (v^{+})^{\beta} dx, 
\end{align*}
and its ground state level 
\begin{equation*}
m_{\infty}:= \inf_{(u, v)\in \X_{0}\setminus \{(0,0)\}} \max_{t\geq 0} \J_{\infty} (tu, tv)>0. 
\end{equation*}
If $\max\{V_{\infty}, W_{\infty}\}= \infty$ we set $m_{\infty}:=\infty$. \\
Since the map $(u,v)\mapsto \int_{\R^{N}} (u^{+})^{\alpha} (v^{+})^{\beta} dx$ is positively $2^{*}_{s}$-homogeneous, the arguments developed in Section $4$ permit to deduce a compactness result for the functional $\J_{\e}$. More precisely, following the lines of the proofs of Theorem \ref{prop2.3} and Corollary \ref{cor2.6}, replacing Lemma \ref{lem2.1} by Lemma \ref{lem4.2}, we can prove that the next result holds.
\begin{thm}\label{prop4.4}
The functional $\J_{\e}$ constrained to $\M_{\e}$ satisfies the $(PS)_{d}$-condition at any level $d<\min\{m_{\infty}, \frac{s}{N} \widetilde{S}_{*}^{\frac{N}{2s}}\}$. Moreover, critical points of $\J_{\e}$ constrained to $\M_{\e}$ are critical points of $\J_{\e}$ in $\X_{\e}$. 
\end{thm} 

\noindent
We conclude this section giving our second multiplicity result. Since many calculations made in Section $5$ can be easily adapted in this context, we present only a sketch of the proof.
\begin{proof}[Proof of Theorem \ref{thm2}]
We proceed as in the proof of Theorem \ref{thm1}. Fix $\delta>0$ and choose $\eta\in C^{\infty}_{0}(\R, [0,1])$ such that $\eta(t)=1$ if $0\leq t\leq \frac{\delta}{2}$ and $\eta(t)=0$ if $t\geq \delta$. Let $(\tilde{w}_{1}, \tilde{w}_{2})\in \X_{0}$ be the solution of \eqref{CP0} given by Theorem \ref{prop4.3}.
For any $y\in M$, we define
\begin{equation*}
\tilde{\Psi}_{i, \e, y}(x):= \eta(|\e x-y|)\tilde{w}_{i}\left(\frac{\e x-y}{\e}\right), \quad i=1,2, 
\end{equation*}
and we introduce the map $\tilde{\Phi}_{\e}(y):=(\tilde{t_{\e}} \tilde{\Psi}_{1, \e, y}, \tilde{t_{\e}}\tilde{\Psi}_{2,\e,y})$, 
where $\tilde{t}_{\e}$ is the unique positive number satisfying
\begin{equation*}
\max_{t\geq 0} \J_{\e}(t \tilde{\Psi}_{1,\e, y}, t \tilde{\Psi}_{2, \e, y})= \J_{\e}(\tilde{t_{\e}} \tilde{\Psi}_{1, \e, y} , \tilde{t}_{\e} \tilde{\Psi}_{2, \e, y}).
\end{equation*} 
As in Section $5$, we can see that
\begin{equation*}
\lim_{\e\rightarrow 0^{+}} \J_{\e}(\tilde{\Phi}_{\e}(y))=m_{0} \, \mbox{ uniformly for } y\in M. 
\end{equation*}

Moreover, denoted by $\varUpsilon: \R^{N}\rightarrow \R^{N}$ the function defined in Section $4$ we can define the barycenter map $\tilde{\beta}_{\e}: \M_{\e}\rightarrow \R^{N}$ given by
\begin{equation*}
\tilde{\beta}_{\e}(u, v):= \frac{\int_{\R^{N}} \varUpsilon(\e x) (|u(x)|^{2} + |v(x)|^{2})\, dx }{\int_{\R^{N}} (|u(x)|^{2} + |v(x)|^{2})\, dx}. 
\end{equation*}

Then it is easy to check that
\begin{equation*}
\lim_{\e\rightarrow 0^{+}} \tilde{\beta}_{\e}(\Phi_{\e}(y))=y \mbox{ uniformly for } y\in M
\end{equation*}
and 
\begin{equation*}
\lim_{\e\rightarrow 0^{+}} \sup_{(u,v)\in \tilde{\Sigma}_{\e}} dist(\tilde{\beta}_{\e}(u, v) , M_{\delta})=0, 
\end{equation*}
where 
$$
\widetilde{\M}_{\e}:=\{(u,v)\in \M_{\e} : \J_{\e}(u, v) \leq m_{0}+ \tilde{h}(\e)\}
$$
and $h:[0, \infty)\rightarrow [0, \infty)$ satisfies $\tilde{h}(\e)\rightarrow 0$ as $\e\rightarrow 0^{+}$. \\
Consequently, there exists $\e_{\delta}>0$ such that for any $\e\in (0, \e_{\delta})$ the diagram
\begin{equation*}
M \stackrel{\tilde{\Phi}_{\e}}{\rightarrow} \widetilde{\M}_{\e} \stackrel{\tilde{\beta}_{\e}}{\rightarrow} M_{\delta}
\end{equation*}
is well defined and $\tilde{\beta}_{\e}\circ \tilde{\Phi}_{\e}$ is homotopically equivalent to the embedding $\iota: M\rightarrow M_{\delta}$. Therefore $cat_{\widetilde{\M}_{\e}}(\widetilde{\M}_{\e})\geq cat_{M_{\delta}}(M)$. From Theorem \ref{prop4.4} and 
$
m_{0}<\frac{s}{N} \widetilde{S}_{*}^{\frac{N}{2s}},
$
we may suppose that $\e_{\delta}$ is so small such that $\J_{\e}$ satisfies the Palais-Smale condition in $\widetilde{\M}_{\e}$. Then the proof goes as in the subcritical case by using Ljusternik-Schnirelmann theory.   
\end{proof}

\end{document}